\documentclass[a4paper,12pt]{amsart}
\usepackage{amsmath,amssymb,amsxtra,graphicx,psfrag}
\usepackage{bm,mathrsfs}
\usepackage{mathtools}
\usepackage{xspace}
\usepackage{color}
\usepackage{xcolor}
\usepackage{hhline}
\usepackage{comment}

\definecolor{forestgreen}{rgb}{0.13, 0.55, 0.13}
\definecolor{lightblue}{rgb}{0.68, 0.85, 0.9}
\usepackage[colorlinks=true,
            linkcolor=blue,
            urlcolor  = blue, 
            citecolor=forestgreen,
            anchorcolor = lightblue]{hyperref}

  \def\MR#1{\href{http://www.ams.org/mathscinet-getitem?mr=#1}{MR#1}}

\usepackage{enumerate}

\usepackage{hyperref}
\usepackage{pdfsync}
\usepackage{hhline}
\usepackage{fancyhdr}
\usepackage{epsfig}
\usepackage{epsfig,subfigure,epstopdf}
\allowdisplaybreaks

\usepackage{pstricks,pst-plot,pst-func}
\usepackage{pspicture}
\usepackage{curves}

\include{rgb}

\def\e{{\rm e}}
\def\i{{\rm i}}
\def\d{{\rm d}}
\def\C{{\mathbb C}}
\def\N{{\mathbb N}}
\def\P {{\mathbb P}}
\def\Re {{\mathbb R}}

\def\K{{\mathscr{K}}}
\def\D{{\mathscr{D}}}

\DeclareMathOperator\Real {Re}

\definecolor{mygray}{rgb}{0.9,0.9,0.9}

\begin{document}
\theoremstyle{plain}
\newtheorem{theorem}{Theorem}[section]
\newtheorem{lemma}{Lemma}[section]
\newtheorem{proposition}{Proposition}[section]
\newtheorem{corollary}{Corollary}[section]

\theoremstyle{definition}
\newtheorem{definition}[corollary]{Definition}

\newtheorem{example}{Example}[section]

\newtheorem{remark}{Remark}[section]
\newtheorem{remarks}[remark]{Remarks}
\newtheorem{note}{Note}
\newtheorem{case}{Case}

\numberwithin{equation}{section}
\numberwithin{table}{section}
\numberwithin{figure}{section}

\title[The WSBDF7 method]
{The weighted and shifted seven-step BDF method\\ for parabolic equations}
\author[Georgios Akrivis]{Georgios Akrivis}
\address{Department of Computer Science and Engineering, University of Ioannina, 451$\,$10 Ioannina, Greece,
and Institute of Applied and Computational Mathematics, FORTH, 700$\,$13 Heraklion, Crete, Greece}
\email {\href{mailto:akrivis@cse.uoi.gr}{akrivis{\it @\,}cse.uoi.gr}}

\author[Minghua Chen]{Minghua Chen}
\address{School of Mathematics and Statistics, Gansu Key Laboratory of Applied Mathematics and Complex Systems,
 Lanzhou University, Lanzhou 730000, P.R. China}
\email {\href{mailto:chenmh@lzu.edu.cn}{chenmh{\it @\,}lzu.edu.cn}}

\author[Fan Yu]{Fan Yu}
\address{School of Mathematics and Statistics, Gansu Key Laboratory of Applied Mathematics and Complex Systems,
 Lanzhou University, Lanzhou 730000, P.R. China}
\email {\href{mailto:yuf17@lzu.edu.cn}{yuf20{\it @\,}lzu.edu.cn}}

%
%
%
%
%
%

\begin{abstract}
Stability of the BDF methods of order up to five for parabolic equations can 
be established by the energy technique via Nevanlinna--Odeh multipliers.
The nonexistence of Nevanlinna--Odeh multipliers makes the 
six-step BDF method special; however, the energy  technique was recently extended 
by the authors in  [Akrivis et al., SIAM J.\ Numer.\ Anal.\  \textbf{59} (2021) 2449--2472]
and covers all six stable BDF methods.
The seven-step BDF method is unstable  for parabolic equations, since it is not even zero-stable.
In this work, we construct and analyze a stable linear combination of two non zero-stable schemes,
the seven-step BDF method and its shifted counterpart, referred to as WSBDF7 method. The stability regions 
of the WSBDF$q, q\leqslant 7$, with a \emph{weight} $\vartheta\geqslant1$, increase as $\vartheta$ increases,
are  larger than the stability regions of the classical BDF$q,$ corresponding to $\vartheta=1$.
We determine  novel  and suitable multipliers
for  the WSBDF7 method and establish stability for parabolic equations by the energy technique.
The proposed approach  is applicable for mean curvature flow,  gradient flows, fractional equations  and  nonlinear equations.
\end{abstract}

\maketitle


\section{Introduction}\label{Se:intro}
Let $T >0, u^0\in H,$ and consider the initial value
problem of seeking $u \in C((0,T];\D(A))\cap C([0,T];H)$ satisfying
\begin{equation}
\label{ivp}
\left \{
\begin{aligned}
&u' (t) + Au(t) = f(t), \quad 0<t<T,\\
& u(0)=u^0 ,
\end{aligned}
\right .
\end{equation}
with  $A$ a positive definite, selfadjoint, linear operator on a
Hilbert space $(H, (\cdot , \cdot )) $ with domain  $\D(A)$
dense in $H$ and $f  : [0,T] \to H$ a given forcing term.
We shall analyze the discretization of \eqref{ivp}
by the  \emph{weighted} and \emph{shifted}
$q$-step  backward difference formula (WSBDF$q$) with $q=7$, described
by a \emph{weight} $\vartheta>0$ and the corresponding characteristic polynomials
$\alpha$ and $\beta,$
\begin{equation}\label{pol:alpha-beta}
\alpha (\zeta):=\vartheta a (\zeta)+(1-\vartheta)\tilde a (\zeta)=\sum_{j=0}^q \alpha_j \zeta^j,
\quad \beta (\zeta):=\vartheta \zeta^q+(1-\vartheta)\zeta^{q-1}, ~ q\leqslant 7,
\end{equation}
with $a$ and $\tilde a$ the characteristic polynomials of the $q$-step BDF method
and the shifted $q$-step BDF method, respectively, for $1\leqslant q\leqslant 7$,
\begin{equation}\label{pol:a-b}
\left\{
\begin{aligned}
&a (\zeta)= \sum_{j=1}^q \frac 1j \zeta^{q-j} (\zeta-1)^j=\sum_{j=0}^q a_j \zeta^j,\\
&\tilde a (\zeta)=a (\zeta)-\sum_{j=2}^q \frac{1}{j-1} \zeta^{q-j} (\zeta-1)^j
=\sum_{j=0}^q \tilde a_j \zeta^j.
\end{aligned}
\right.
\end{equation}
In particular, $\tilde a (\zeta)=a (\zeta)$ for $q=1$.

Let $N\in \N, \tau:=T/N$ be the time step, and $t_n :=n \tau,
n=0,\dotsc ,N,$ be a uniform partition of the interval $[0,T].$
We recursively define a sequence of approximations $u^m$ to
the nodal values $ u(t_m)$ of the exact solution by the WSBDF7 method,
\begin{equation}\label{ab}
\sum_{i=0}^7 \alpha_i u^{n+i}+\vartheta\tau A u^{n+7}+(1-\vartheta)\tau A u^{n+6}=\vartheta\tau f^{n+7}+(1-\vartheta)\tau f^{n+6}
\end{equation}
for $n=0,\dotsc,N-7,$ with $f^m:=f(t_{m}),$ assuming that starting approximations $u^0, \dotsc, u^{6}$ are given.
For convenience, we suppressed the dependence of $\alpha$ and of its coefficients on $\vartheta.$
We are particularly interested in the WSBDF7 method \eqref{ab} for $\vartheta=3,$
\begin{equation}\label{ab-3}
\sum_{i=0}^7 \alpha_i u^{n+i}+3\tau A u^{n+7}-2\tau A u^{n+6}=3\tau f^{n+7}-2\tau f^{n+6},
\quad n=0,\dotsc,N-7.
\end{equation}

Let $P_7\in \P_7$ be the Lagrange interpolating polynomial of a function $y$ at the nodes $t_{n},t_{n+1},\dotsc, t_{n+7}.$
We recall that the seven-step BDF method,
\begin{equation}\label{BDF7}
\sum_{i=0}^7 a_i y^{n+i}=\tau f (t_{n+7},y^{n+7}),
\end{equation}
for an o.d.e.\ $y'=f(t,y)$, is constructed by approximating the derivative
of $y$ at the node $t_{n+7}$ in the relation $y'(t_{n+7})=f\big (t_{n+7},y(t_{n+7})\big )$ by the derivative
$P_7'(t_{n+7})$ of the interpolating polynomial. Analogously, the shifted seven-step BDF method,
\begin{equation}\label{SBDF7}
\sum_{i=0}^7 \tilde a_i y^{n+i}=\tau f (t_{n+6},y^{n+6}),
\end{equation}
is constructed by approximating $y'(t_{n+6})$ in the relation $y'(t_{n+6})=f\big (t_{n+6},y(t_{n+6})\big )$ by
$P_7'(t_{n+6}).$ Notice that $\vartheta P_7'(t_{n+7})+(1-\vartheta)P_7'(t_{n+6})$ is, in general,
different from $P_7'(t_{n+6}+\vartheta\tau)$.

Multiplying \eqref{BDF7} and \eqref{SBDF7} by $\vartheta$ and $1-\vartheta,$ respectively, and adding the
results, we obtain the weighted and shifted BDF7  (WSBDF7) method,
\begin{equation}\label{WSBDF7}
\vartheta \sum_{i=0}^7  a_i y^{n+i}+(1-\vartheta)\sum_{i=0}^7 \tilde a_i y^{n+i}=\vartheta \tau f (t_{n+7},y^{n+7})
+(1-\vartheta) \tau f (t_{n+6},y^{n+6}).
\end{equation}
In particular, for $\vartheta=1$ and $\vartheta=0,$ the WSBDF7 method reduces to the standard seven-step BDF method
and to the corresponding SBDF method.

It is well known that both methods, BDF7 and SBDF7, are not zero-stable; for BDF7, see, e.g., \cite[Theorem 3.4]{HNW};
concerning SBDF7, it is easily seen that $\tilde a (-13)\tilde a (-12)<0,$ whence $\tilde a$ has a root in the interval
$(-13,-12).$
The order of both methods is $7.$
Here, we show that their combination, WSBDF7, is $A(\varphi)$-stable for $\vartheta=3$, stable even for parabolic equations.

The explicit form of the polynomials $a$ and $\tilde a$ in \eqref{pol:a-b} with $q=7$  is
\begin{equation*}
\begin{split}
a (\zeta)&=\frac{1}{420}\big (1089\zeta^7-2940\zeta^6+4410\zeta^5-4900\zeta^4+3675\zeta^3-1764\zeta^2+490\zeta-60\big),\\
\tilde a (\zeta)&=\frac{1}{420}\big( 60\zeta^7+609\zeta^6-1260\zeta^5+1050\zeta^4-700\zeta^3+315\zeta^2- 84\zeta+10\big).
\end{split}
\end{equation*}

Let $| \cdot |$  denote the norm on $H$ induced by the inner product $(\cdot , \cdot )$, and introduce on $V, V:=\D(A^{1/2}),$
the norm $\| \cdot \|$   by $\| v\| :=| A^{1/2} v |.$
We identify $H$ with its dual, and denote by $V'$ the dual of $V$,
and by $\| \cdot \|_\star$ the dual norm on $V', \|v \|_\star=| A^{-1/2} v |.$
We shall use the notation $(\cdot , \cdot )$ also for the antiduality
pairing between $V'$ and $V.$ For simplicity, we denote by $\langle\cdot,\cdot\rangle$ the inner product
on $V,$ $\langle v, w\rangle:=(A^{1/2}v, A^{1/2}w).$

Our  stability results are established by the energy technique utilizing suitable multipliers, and are given in 
the following two theorems and in a corollary.

\begin{theorem}[Stability of method \eqref{ab-3}]\label{Th:stab}
Let $u^0, u^1, \dotsc, u^6\in V.$  The WSBDF7 method \eqref{ab-3} is stable in the sense that 
\begin{equation}\label{stab-abg2}
|u^n|^2+\tau\|u^n\|^2
 \leqslant C\sum_{j=0}^6 \big (|u^j|^2+\tau\|u^j\|^2\big )+C \tau\sum_{\ell=6}^n\| f^{\ell}\|_\star^2,\quad n=7,\dotsc,N.
\end{equation}
Here $C$ denotes a generic constant, 
independent  of $T$ and the operator $A$ as well as of $f, \tau,$ and $n$.
\end{theorem}

\begin{theorem}[Second stability estimate]\label{Th:stab2}
Let $u^0, u^1, \dotsc, u^6\in V,$ and let us indicate by a dot the application of the $7$-step weighted 
and shifted backward difference operator,
\begin{equation}
\label{BDF-oper}
\dot v^m:=\frac 1\tau\sum_{i=0}^7 \alpha_i v^{m-7+i},\quad m=7,\dotsc,N.
\end{equation}
The WSBDF7 method  \eqref{ab-3} is stable in the sense that 
\begin{equation}
\label{stab-abg2-n}
\|u^n\|^2 + \tau  |\dot u^n|^2
\leqslant C\sum_{j=0}^6 \|u^j\|^2 
+C\tau \sum_{\ell=6}^n |f^\ell|^2,\quad n=7,\dotsc,N.
\end{equation}
Here $C$ denotes a generic constant as in Theorem \ref{Th:stab}.
\end{theorem}

\begin{corollary}[Third stability estimate]\label{Co:stab}
Let $u^0, u^1, \dotsc, u^6\in V,$ and let us denote by $\partial_\tau u^n:=(u^n-u^{n-1})/\tau,
n=1,\dotsc,N,$ the backward difference quotients.
The WSBDF7 method  \eqref{ab-3} is stable in the sense that 
\begin{equation}
\label{stab-abg2-nn}
\|u^n\|^2 + \tau  |\partial_\tau u^n|^2
\leqslant C\sum_{j=0}^6 \|u^j\|^2 +C\tau\sum_{j=1}^6 |\partial_\tau u^j|^2 
+C\tau \sum_{\ell=6}^n |f^\ell|^2,\quad n=7,\dotsc,N,
\end{equation}
with a generic constant $C.$ 
\end{corollary}

The application of the energy technique to establish stability of high order multistep methods for parabolic equations 
relies on  suitable multipliers. The multiplier technique was introduced by Nevanlinna and Odeh in \cite{NO}
and is based on Dahquist's equivalence between A- and G-stability; see also \cite[\textsection V.8, pp.\ 342--347]{HW}.
In \cite{NO}, suitable multipliers for the three-, four- and five-step BDF methods were determined;
see also  \cite{AK:16} for optimal Nevanlinna--Odeh multipliers for these methods, i.e., multipliers with
minimal sum of  absolute values.

The multiplier technique became widely known and popular after its first actual application
to the stability analysis  for parabolic equations by  Lubich, Mansour, and Venkataraman
in 2013; see \cite{LMV}.  In recent years, the energy technique has been frequently used 
in the analyses of various variants of BDF methods of order up to $5,$ such as fully implicit, 
linearly implicit or implicit--explicit, for a series of linear and nonlinear equations of parabolic type.
Nonexistence of Nevanlinna--Odeh multipliers for the six-step BDF method
was established in  \cite{ACYZ:21}; there, to overcome this difficulty, the notion of multipliers 
was slightly modified, and  multipliers for the six-step BDF method were determined, which,  
in combination with the Grenander--Szeg\H{o} theorem for symmetric banded Toeplitz matrices,
made the energy technique applicable also to this method for parabolic equations
with self-adjoint elliptic part.

Here, focusing on the discretization  of parabolic equations with self-adjoint elliptic part by multistep
methods, we first extend the notion of multipliers, and then determine suitable multipliers 
for the WSBDF7  method \eqref{ab-3}  and prove Theorems \ref{Th:stab} and \ref{Th:stab2}
by the energy technique.
The new, more general notion of multipliers reduces to the corresponding notion in  \cite{ACYZ:21}
in the case of the BDF methods; however, both the proofs and the 
stability results here and in  \cite{ACYZ:21} are different.
The present approach is shorter and simpler but it yields  weaker stability results,
in the sense that   \eqref{stab-abg2} leads to optimal order
error estimates in the discrete $\ell^\infty(H)$ norm but to suboptimal by half-an-order
error estimates in the discrete $\ell^\infty(V)$ norm; see \eqref{conv3};
in contrast, the stability estimates in  \cite{ACYZ:21}  lead to optimal order
error estimate in the discrete $\ell^\infty(H)$ as well as 
in the discrete $\ell^2(V)$ norms. Of course, \eqref{stab-abg2-n} leads to optimal order
error estimates in the discrete $\ell^\infty(V)$ norm.
Let us also mention that the stability
approach in \cite{ACYZ:21} is restricted to BDF methods, in which
case banded Toeplitz matrices enter. In the case of the
WSBDF7  method \eqref{ab-3}, the corresponding matrices reduce to
banded Toeplitz matrices only after discarding their last row and column;
this fact prevents us from using the Grenander--Szeg\H{o} theorem.

In 1991, linear multistep methods of orders from $2$ to $7$ 
for ordinary differential equations, with stability 
regions larger than the stability regions  of the BDF method of the same order, 
were constructed   in \cite{LX:1991}.
In particular, the seven-step methods of order $7$ of  \cite{LX:1991}
are the WSBDF7  method \eqref{ab} with a parameter $\vartheta.$
High order implicit-explicit multistep  methods were constructed and
analyzed in \cite{Cr:1980}. Then, in 1995, implicit-explicit multistep schemes of orders
from $1$ to $4$ were constructed in \cite{ARW:95}; the method of order $2$ in
\cite{ARW:95} coincides with the WSBDF2 method, but  the methods of order $3$
and $4$ are different from the WSBDF$q, q=3,4,$ methods.
The construction techniques in \cite{LX:1991} and for the WSBDF methods
are significantly different.  The point of departure in  \cite{LX:1991} is the polynomial 
$\beta$ in \eqref{pol:alpha-beta}; the corresponding polynomial $\alpha$
 is then determined via the order conditions.
 Here, the WSBDF7 method \eqref{WSBDF7}  is constructed by the simpler, direct, and more efficient
 weighted and shifted technique, which immediately extends also to variable time step schemes.
For example, 
the variable time-step  WSBDF3 methods for parabolic equations
are presented in \cite{CYZZ:21}. However, there is no published work 
for variable or even uniform time-step  WSBDF$q, q\geqslant4,$ methods.

The proposed methods,  for $q\leqslant 6,$ including the classical case $\vartheta=1$, 
have been recently widely used to various applied scientific phenomena, such as 
 mean curvature flow \cite{EGK,KL},  gradient flows \cite{HS}, and fractional equations \cite{CYZ}.

An outline of the paper is as follows: In the short Section \ref{Se:stab-regions}, 
we briefly comment on the stability regions of the WSBDF7  method \eqref{WSBDF7}
for various values of the parameter $\vartheta.$
In Section \ref{Se:mult}, we make precise the requirements on the multipliers
for multistep methods that are suitable for our stability approach,
and show that our notion extends the multiplier notion of \cite{ACYZ:21}
for BDF methods. The remaining part of the article is devoted to the WSBDF7  
method \eqref{ab-3}. First, in Section \ref{Se:WSBDF7-mult}
we give a suitable multiplier for this method
and in Section \ref{Se:det-mult} comment on the determination of such
multipliers; in particular, we provide information about the range of
such multipliers with up to the first four nonvanishing components.
Sections \ref{Se:stab} and \ref{Se:err-est} are devoted
to the proof of the stability estimates \eqref{stab-abg2} and \eqref{stab-abg2-n},
and to the derivation of error estimates.  We conclude in Section \ref{Se:numerics}
with numerical results.

\section{Stability regions}\label{Se:stab-regions}

The BDF$q$ methods are $A(\varphi_q)$-stable with $\varphi_1=\varphi_2=90^\circ$,
$\varphi_3\approx 86.03^\circ, \varphi_4\approx 73.35^\circ, \varphi_5\approx 51.84^\circ,$ and $\varphi_6 \approx 17.84^\circ$;
see \cite[Section V.2]{HW}.
The WSBDF$q$ methods are $A(\tilde \varphi_q)$-stable with $\tilde \varphi_1=\tilde \varphi_2=90^\circ$,
$\tilde\varphi_3\approx 89.55^\circ, \tilde\varphi_4\approx 85.32^\circ, \tilde\varphi_5\approx 73.2^\circ,$ and $\tilde\varphi_6 \approx 51.23^\circ,$  
for the weights  $\vartheta_3=20,\vartheta_4=60,\vartheta_5=48,\vartheta_6=50,$ respectively; see \cite[Table 2]{LX:1991}.
Notice that  $\varphi_q<\tilde\varphi_q<\varphi_{q-1}$ for $q=3,4,5,6.$ It can also be shown that $\tilde\varphi_q\to \varphi_{q-1}$ as $\vartheta\to \infty$
for $q=3,\dotsc,7;$ see Remark  \ref{Re:angles}. It seems that the value $\tilde \varphi_7 \approx 18.32^\circ>\varphi_6 \approx 17.84^\circ$
for $\vartheta_7=200/7$  given in \cite{LX:1991} is incorrect. 
We numerically computed the approximations
$\tilde \varphi_1=\tilde \varphi_2=90^\circ, \tilde\varphi_3\approx 89.99^\circ, \tilde\varphi_4\approx 85.93^\circ, \tilde\varphi_5\approx 73.2^\circ$, 
$\tilde \varphi_6 \approx 51.63^\circ,$  and $\tilde \varphi_7 \approx 17.47^\circ$ for the weight  $\vartheta=100$.

\begin{remark}[The limit of the stability angles $\tilde\varphi_q$ of WSBDF$q$]\label{Re:angles}
N\o rsett,  \cite[p.\ 263]{N}, established an  A$(\varphi_q)$-stability criterion for  the BDF$q$ methods, namely, 
in his notation,
\begin{equation*}
\tan(\varphi_{q})=\min_{x\in D_q}\left(-\sqrt{1-x^2}\cdot \frac{ I_q(x)}{R_q(x)}\right),\quad D_q=\{x\in[-1,1], R_q(x)<0\}; 
\end{equation*}
here $I_q(x)$ and $R_q(x)$ are related to the imaginary and real parts of points on the root locus curve of the method,
and are expressed in terms of the Chebyshev polynomials.
Using this  criterion and  results in \cite[p.\ 7]{LX:1991}, we have 
\begin{equation*}
\tan(\tilde\varphi_{q})=\min_{x\in \tilde D_q}\left(-\sqrt{1-x^2}\cdot\lim_{\vartheta\rightarrow \infty}\frac{ \tilde I_q(x)}{\tilde R_q(x)}\right),\quad \tilde D_q=\{x\in[-1,1], \tilde R_q(x)<0\}
\end{equation*}
for the stability angle $\tilde \varphi_q$ of the WSBDF$q$ method.
Now, we can easily check that 
\[\lim_{\vartheta\rightarrow \infty}   \frac{\tilde I_q(x)}{\tilde R_q(x)}=\frac{I_{q-1}(x)}{R_{q-1}(x)},\quad q=3,\dotsc,7,\]
and infer that 
\[\tilde\varphi_q\rightarrow\varphi_{q-1}\quad\text{as }\ \vartheta\to \infty, \quad q=3,\dotsc,7.\]
\end{remark}

Here, we examine the stability regions  of the WSBDF7 method \eqref{WSBDF7}.
For the test equation $y'(t)=\lambda y(t), \lambda\in \C$, the method reads
\begin{equation}\label{abb}
\sum_{i=0}^7 \alpha_i y^{n+i}=\lambda\tau \vartheta  y^{n+7}+\lambda\tau(1-\vartheta)y^{n+6}, \quad n=0,\dotsc,N-7.
\end{equation}
To study the stability regions, we set  $z=\lambda\tau$ in \eqref{abb} and consider the characteristic polynomial
\begin{equation*}\label{abbb}
p(\zeta):=\alpha(\zeta)-z\beta(\zeta),\quad \zeta\in \C;
\end{equation*}
again, for convenience, we suppressed the dependence of the polynomials $\alpha$ and $\beta$
on $\vartheta;$ see \eqref{pol:alpha-beta}.
Then, the stability region of the method  is the set of all $z\in\mathbb{C}$ such that the characteristic
polynomial $p$ satisfies the root condition.
We plot the stability regions of the WSBDF7 method \eqref{abb} for $\vartheta=1,3,10$ in Figure \ref{Fig:region}.
Notice that the stability regions increase as $\vartheta$ increases.

\begin{figure}[htb]  \centering
  \begin{tabular}{cc}
      \includegraphics[width=5.2cm]{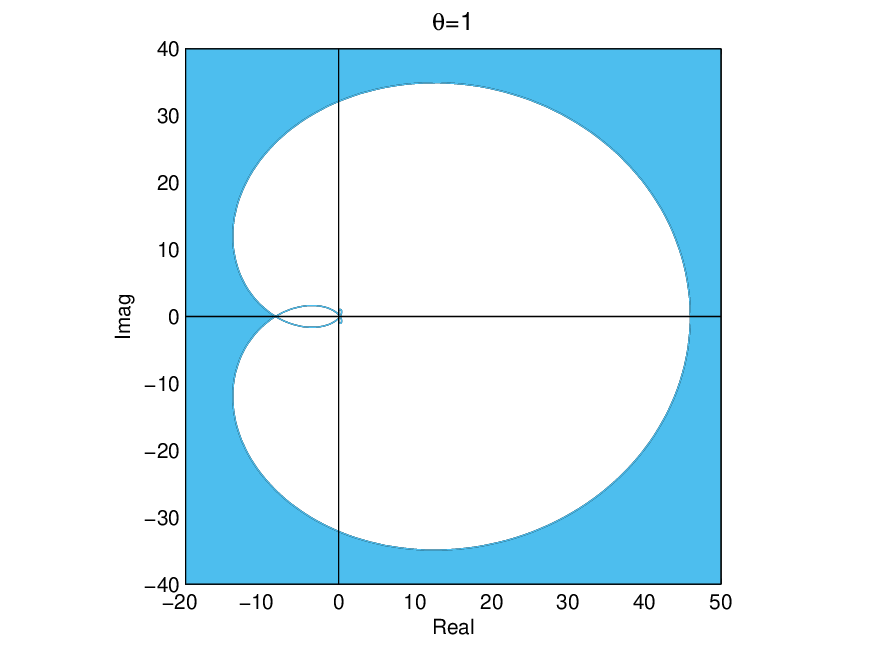}~~
      \includegraphics[width=5.2cm]{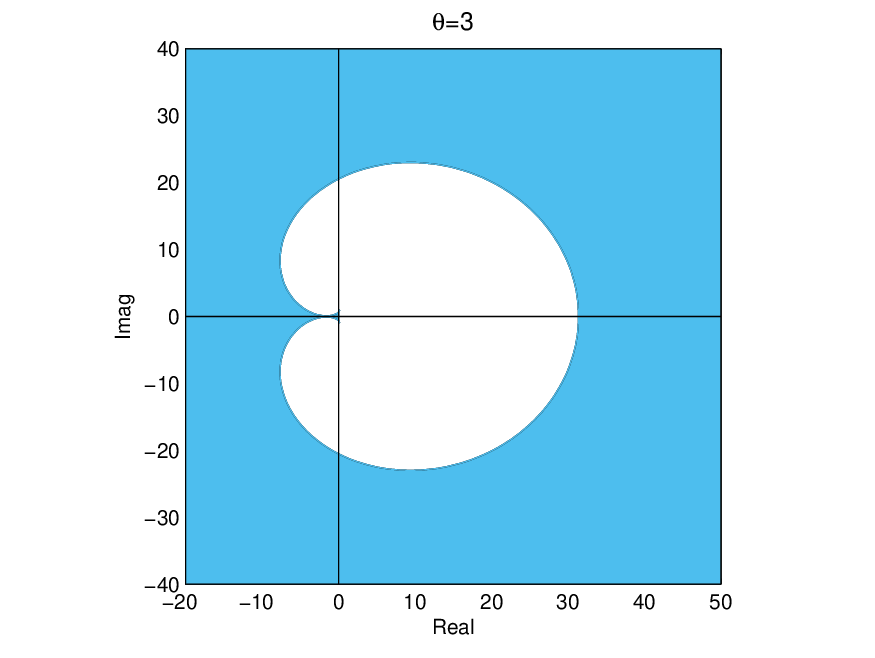}~~
      \includegraphics[width=5.2cm]{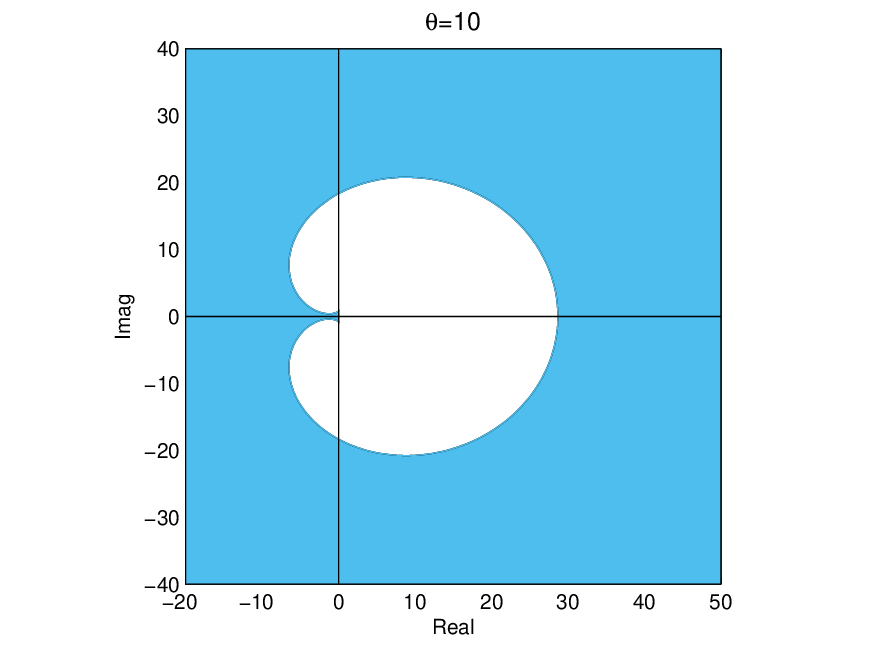}~~\\
      \includegraphics[width=5.2cm]{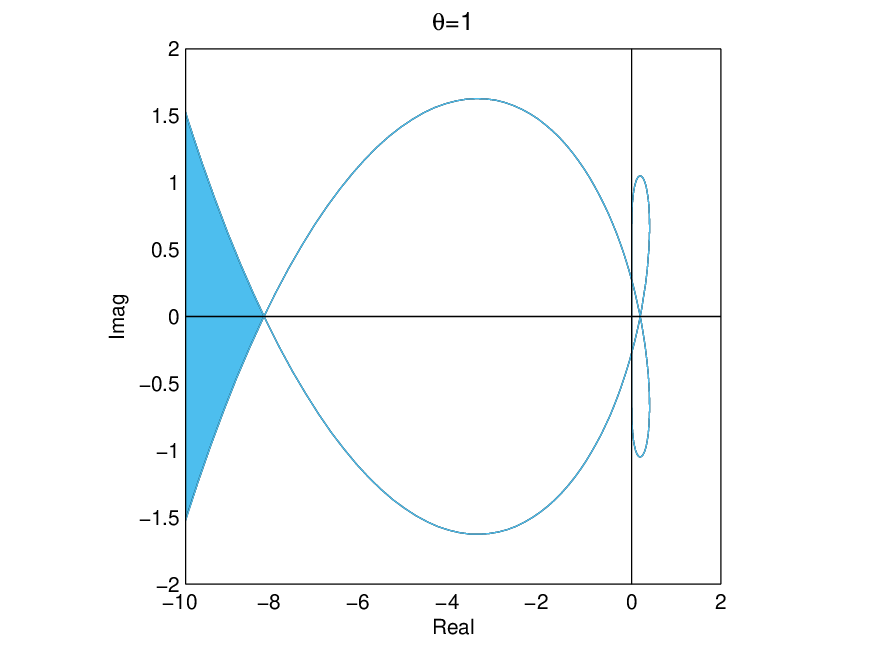}~~
      \includegraphics[width=5.2cm]{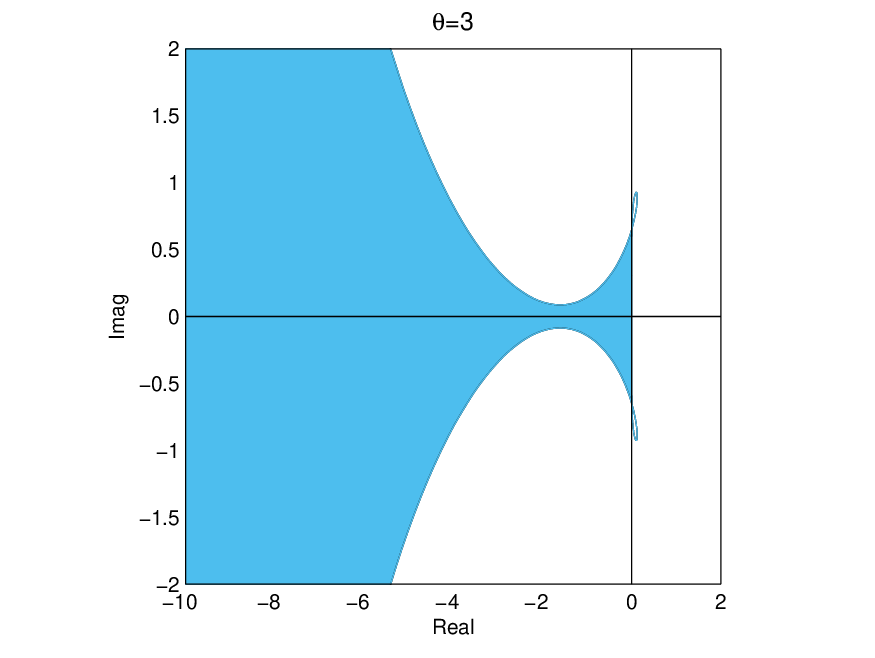}~~
      \includegraphics[width=5.2cm]{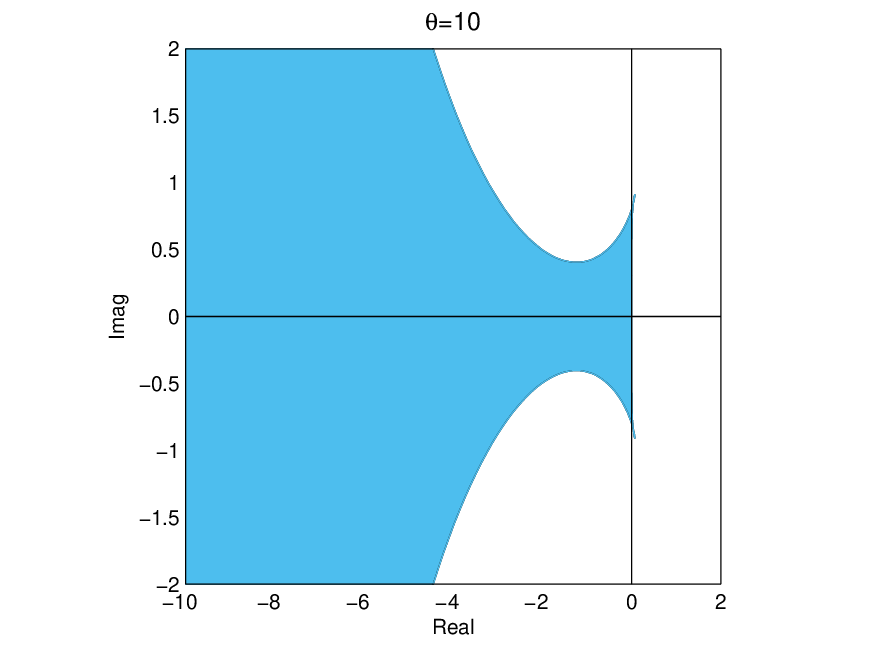}~~
  \end{tabular}
  \caption{The stability  regions $($upper panel$)$ as well as zoom in around the origin  $($lower panel$)$, in light blue, for $\vartheta=1,3,10$, respectively.}
  \label{Fig:region}
\end{figure}

\section{Multipliers for multistep methods}\label{Se:mult}
Here, we extend the notion of multipliers for  multistep methods
applied to parabolic equations with self-adjoint elliptic part.
The present, more general notion of multipliers reduces to the corresponding 
notion in  \cite{ACYZ:21} in the case of the BDF methods. 
The multiplier  technique hinges on the celebrated equivalence
 of A- and G-stability for multistep methods by Dahlquist.

\begin{lemma}[\cite{D}; see also \cite{BC} and {\cite[Section V.6]{HW}}]\label{Le:Dahl}
Let $\rho(\zeta)=\rho_q\zeta^q+\dotsb+\rho_0$ and
$\sigma(\zeta)=\sigma_q\zeta^q+\dotsb+\sigma_0$ be polynomials of degree  $q,$
with real coefficients, 
that have no common divisor.
Let $(\cdot,\cdot)$ be a real inner product with associated norm $|\cdot|.$
If
\begin{equation}
\label{A}
\Real \frac {\rho(\zeta)}{\sigma(\zeta)}>0\quad\text{for }\, |\zeta|>1,
\tag{A}
\end{equation}
then there exists a positive definite symmetric matrix $G=(g_{ij})\in \Re^{q,q}$
and real $\gamma_0,\dotsc,\gamma_q$ such that for $v^0,\dotsc,v^{q}$ in the inner product space,
\begin{equation}
\label{G}
 \Big (\sum_{i=0}^q\rho_iv^{i},\sum_{j=0}^q\sigma_jv^{j}\Big )=
\sum_{i,j=1}^qg_{ij}(v^{i},v^{j})
-\sum_{i,j=1}^qg_{ij}(v^{i-1},v^{j-1})
+\Big |\sum_{i=0}^q\gamma_iv^{i}\Big |^2.   
\tag{G}
\end{equation}
\end{lemma}
Let us briefly comment on the assumption that $\rho$ and
$\sigma$ are polynomials of the same degree; obviously, if \eqref{A} 
would be satisfied for polynomials of different degrees, then 
explicit A-stable methods would exist. First, nonconstant polynomials
cannot satisfy \eqref{A} since they cannot pertain the sign of their real part as $|z|\to \infty.$
This shows also that the degree of $\rho$ cannot exceed the degree of $\sigma,$
since then we could write  $\rho/\sigma$ as the sum of
a nonconstant polynomial and a rational function $R$ such that
the degree of its numerator is lower than the degree of its denominator, whence, in particular,
$\lim_{|z|\to \infty}\Real R(z)=0.$ 
Finally, \eqref{A} is symmetric with respect to $\rho$ and $\sigma$ since
$\Real [\rho(z)/\sigma(z)]\Real  [\sigma(z)/\rho(z)] =\cos^2\varphi >0$
for $\rho(z)/\sigma(z)=r\e^{\i \varphi}$ not purely imaginary.

Next, we specify our requirements on the multipliers; in Section \ref{Se:WSBDF7-mult}, we shall provide
motivation for these  requirements.

\begin{definition}[Multipliers]\label{De:multipliers} 
{\upshape Let $\alpha$ and $\beta$ be the characteristic polynomials of an A$(0)$-stable $q$-step method;
then, $\alpha_q\beta_q>0,$ and thus we can assume that the leading coefficients $\alpha_q$ 
and $\beta_q$ are positive. For a $q$-tuple of real numbers
$(\mu_1,\dotsc,\mu_q)$ consider the polynomial $\mu(\zeta):=\zeta^q-\mu_1\zeta^{q-1}-\dotsb -\mu_q.$
Then, $(\mu_1,\dotsc,\mu_q)$  is called a \emph{multiplier} of the method if it satisfies three properties, 
namely, if the pairs of polynomials
$(\alpha,\mu)$ and $(\beta,\mu)$ have no common divisors, except possibly of a common factor 
of the form $\zeta^\ell$ for the pair $(\beta,\mu)$, and satisfy the
A-stability condition \eqref{A} in Lemma \ref{Le:Dahl} and a slightly more restrictive version of it,
respectively, that is,
\begin{equation}
\label{A-alpha-mu}
\Real \frac {\alpha(\zeta)}{\mu(\zeta)}>0\quad\text{for }\, |\zeta|>1,
\end{equation}
i.e., the method described by the coefficients of the polynomials $\alpha$ and $\mu$ is $A$-stable, and
\begin{equation}
\label{A-beta-mu}
\Real \frac {\beta(\zeta)}{\mu(\zeta)}>c_\star\quad\text{for }\, |\zeta|>1,
\end{equation}
for some positive constant $c_\star,$ and the polynomial $\mu$ does not have unimodular roots.
}
\end{definition}

The A-stability condition \eqref{A-alpha-mu} is symmetric with respect to the polynomials $\alpha$ and $\mu,$
since for any not purely imaginary complex number $z=r\e^{\i \varphi},$ we have
$\Real z\Real\frac 1z =\cos^2\varphi >0.$ This property is crucial because otherwise the A-stability condition 
\eqref{A} would not be equivalent to the obviously symmetric condition \eqref{G}.
On the other hand, the more stringent condition \eqref{A-beta-mu} cannot be symmetric in case
$\mu$ has unimodular roots, since it implies that $\beta$ does not have unimodular roots.
For instance for $\mu(\zeta)=\zeta-1$ and $\beta(\zeta)=\zeta,$ i.e., for the characteristic
polynomials of the implicit Euler method, we obviously have
\[\lim_{\zeta\to 1}\frac {\mu(\zeta)}{\beta(\zeta)}=0,\]
while, with $\zeta=r\e^{\i \varphi},$
\[\Real \frac {\beta(\zeta)}{\mu(\zeta)}=\frac {r^2-r\cos\varphi}{r^2-2r\cos\varphi+1}
\quad\text{and}\quad \frac {r^2-r\cos\varphi}{r^2-2r\cos\varphi+1}\geqslant \frac 12
\iff r\geqslant 1.\]
The additional condition that $\mu$ does not have unimodular roots
makes condition \eqref{A-beta-mu} symmetric; cf.\ also  Lemma \ref{Le:trig-inequalities}.

Let us note that the third property of multipliers, i.e., that $\mu$ does not have unimodular roots,
is not needed for the proof of Theorem \ref{Th:stab}; we will use it only in the proof of
Theorem \ref{Th:stab2}.

\begin{remark}[Equivalent version of the conditions on the pair $(\beta,\mu)$]\label{Re:delta-kappa}
Let $\ell\geqslant 0$ be the largest integer such that $\zeta^\ell$ is a common factor of 
$\beta$ and $\mu.$ We factor $\zeta^\ell$ out, and write $\beta$ and $\mu$ in the form
\begin{equation}
\label{A-beta-mu-delta-kappa}
\beta (\zeta)=\zeta^\ell\delta (\zeta)\quad\text{and}\quad \mu (\zeta)=\zeta^\ell\kappa (\zeta).
\end{equation}
Then, the conditions on the pair  $(\beta,\mu)$ in the previous Definition can be equivalently 
formulated in the form: the pair of polynomials $(\delta,\kappa)$ has no common divisor and satisfies the condition 
\begin{equation}
\label{A-delta-kappa}
\Real \frac {\delta(\zeta)}{\kappa(\zeta)}>c_\star\quad\text{for }\, |\zeta|>1,
\end{equation}
for some positive constant $c_\star.$
\end{remark}

The following result is an immediate consequence of the maximum principle for harmonic functions.

\begin{lemma}[Equivalent versions of  conditions \eqref{A-alpha-mu} and \eqref{A-beta-mu}]\label{Le:trig-inequalities}
With the notation in Definition \ref{De:multipliers} and $\mu_0:=-1,$ conditions   \eqref{A-alpha-mu} and \eqref{A-beta-mu} 
are equivalent to
\begin{equation}
\label{trig-ineq1}
-\sum_{j,\ell=0}^q\alpha_j\mu_\ell \cos ((j+\ell-q)\varphi) \geqslant 0 \quad \forall \varphi \in \Re
\end{equation}
and,  if the polynomial $\mu$ does not have unimodular roots, 
\begin{equation}
\label{trig-ineq2}
-\sum_{j,\ell=0}^q\beta_j\mu_\ell \cos ((j+\ell-q)\varphi) \geqslant \tilde c_\star \quad \forall \varphi \in \Re,
\end{equation}
for some positive constant $\tilde c_\star,$ respectively.
\end{lemma}

\begin{proof}
The rational functions $\alpha/\mu$ and  $\beta/\mu$
are holomorphic outside the unit disk in the complex plane  and
\[\lim_{|z|\to \infty}\frac {\alpha(z)}{\mu (z)}=\alpha_q>0,\quad \lim_{|z|\to \infty}\frac {\beta(z)}{\mu (z)}=\beta_q>0.\]
Therefore, according to the maximum principle for harmonic functions,
\eqref{A-alpha-mu} and \eqref{A-beta-mu} for $c_\star\leqslant \beta_q$ are equivalent to
\begin{equation*}
\Real \big [\alpha(\zeta)\mu(\bar \zeta)\big ]\geqslant 0 \quad \forall \zeta\in \K,
\quad\text{and}\quad \Real \frac {\beta(\zeta)}{\mu(\zeta)}\geqslant c_\star\quad \forall \zeta\in \K,
\end{equation*}
respectively, with $\K$ the unit circle in the complex plane, $\K:=\{\zeta\in \C : |\zeta|=1\},$
i.e., equivalent to
\begin{equation}\label{Real1n}
\Real \big [\alpha(\e^{\i \varphi})\mu (\e^{-\i \varphi})\big ]\geqslant 0 
\quad\text{and}\quad \Real \big [\beta(\e^{\i \varphi})\mu (\e^{-\i \varphi})\big ]\geqslant c_\star |\mu (\e^{\i \varphi})|^2
\quad \forall \varphi \in \Re.
\end{equation}
Now, $\mu(\zeta)=-\sum_{\ell=0}^q\mu_\ell\zeta^{q-\ell},$ and thus
\[\beta(\e^{\i \varphi})\mu (\e^{-\i \varphi})=-\sum_{j=0}^q\beta_j\e^{j\i \varphi}\sum_{\ell=0}^q\mu_\ell\e^{(\ell-q)\i \varphi}
=-\sum_{j,\ell=0}^q\beta_j\mu_\ell\e^{(j+\ell-q)\i \varphi},\]
whence
\[ \Real \big [\beta(\e^{\i \varphi})\mu (\e^{-\i \varphi})\big ]=-\sum_{j,\ell=0}^q\beta_j\mu_\ell\cos ((j+\ell-q)\varphi),\]
and, since $ |\mu (\e^{\i \varphi})|^2$ is bounded from above and below by positive constants, \eqref{trig-ineq2} 
is equivalent to the second relation in \eqref{Real1n}. Analogously, the first relation in \eqref{Real1n}
is equivalent to \eqref{trig-ineq1}.
\end{proof}

\begin{remark}[BDF methods]\label{Re:BDF-mult}
In the case of the standard $q$-step BDF method, 
we have $\beta(\zeta)=\zeta^q,$ whence \eqref{trig-ineq2} takes the form 
\[-\sum_{\ell=0}^q\mu_\ell \cos (\ell\varphi) \geqslant \tilde c_\star \quad \forall \varphi \in \Re,\]
i.e., since $\mu_0=-1,$
\begin{equation}
\label{trig-ineq-BDF}
1-\mu_1\cos \varphi-\dotsb-\mu_q\cos (q\varphi) \geqslant  \tilde c_\star \quad \forall \varphi \in \Re.
\end{equation}
Since  $ \tilde c_\star$ can be chosen arbitrarily small, \eqref{trig-ineq-BDF} can be written as a positivity condition,
\begin{equation}
\label{trig-ineq-BDF2}
1-\mu_1\cos \varphi-\dotsb-\mu_q\cos (q\varphi) >0 \quad \forall \varphi \in \Re.
\end{equation}

Conditions \eqref{A-alpha-mu}, with $\alpha$ the characteristic polynomial of the 
$q$-step BDF method, and \eqref{trig-ineq-BDF2} were used in \cite{ACYZ:21}
to establish stability of BDF methods for parabolic equations with self-adjoint
elliptic part by the energy technique.
The motivation for the positivity condition \eqref{trig-ineq-BDF2}  in \cite{ACYZ:21}
was that $1-\mu_1\cos \varphi-\dotsb-\mu_q\cos (q\varphi)$ was the generating 
function of symmetric banded Toeplitz matrices of arbitrary dimension entering into 
the stability analysis, and  \eqref{trig-ineq-BDF2} ensured the positive definiteness 
of these matrices.
\end{remark}

\section{Multipliers for the WSBDF7 method \eqref{ab-3}}\label{Se:WSBDF7-mult}
From now on, we consider the WSBDF7 method \eqref{ab-3}, i.e., \eqref{ab} with $\vartheta=3.$

As we shall see, 
\begin{equation}
\label{mu2}
\mu_1=1.6,\quad \mu_2=-1.6,\quad \mu_3=1.1,\quad \mu_4=-0.3,\quad\mu_5=\mu_6=\mu_7=0,
\end{equation}
are multipliers of the method with $\tilde c_\star:=0.01$ in \eqref{trig-ineq2}. 

Let us now use these specific multipliers to motivate our requirements in Definition \ref{De:multipliers}.

To prove the first stability result for method \eqref{ab-3}  by the energy technique, we subtract and add the term
$c_{\star}A\Big(u^{n+7}-\sum_{j=1}^4\mu_j u^{n+7-j}\Big)$ from and to its left-hand side, and subsequently
test by $u^{n+7}-\mu_1u^{n+6}-\dotsb-\mu_7u^n=u^{n+7}-\mu_1u^{n+6}-\dotsb-\mu_4u^{n+3}$ to obtain
\begin{equation}
\label{ab-energy}
\Big (\sum_{i=0}^7 \alpha_i u^{n+i},u^{n+7}-\sum_{j=1}^4 \mu_ju^{n+7-j}\Big )
+\tau A_{n+7}+\tau c_{\star}\Big\|u^{n+7}-\sum_{j=1}^4\mu_j u^{n+7-j}\Big\|^2= \tau F_{n+7},
\end{equation}
$n=0,\dotsc,N-7,$ with
\begin{equation*}\label{ab-energy2}
\left\{
\begin{aligned}
A_{n+7}&:=\Big \langle 3 u^{n+7}-2 u^{n+6}
-c_{\star}\Big(u^{n+7}-\sum_{j=1}^4\mu_j u^{n+7-j}\Big),u^{n+7}-\sum_{j=1}^4 \mu_j u^{n+7-j}\Big \rangle,\\
F_{n+7}&:=\Big (3 f^{n+7}-2 f^{n+6},u^{n+7}-\sum_{j=1}^4 \mu_j u^{n+7-j}\Big ).
\end{aligned}
\right.
\end{equation*}
The term $F_{n+7}$ in \eqref{ab-energy} can be easily estimated
from above via elementary inequalities. The term involving the 
approximate solutions will be absorbed in the third term on the left-hand
side of \eqref{ab-energy};  this is the motivation
for the use of a positive constant $c_\star$ in \eqref{A-delta-kappa}
or in  \eqref{A-beta-mu}.

To estimate the first term in \eqref{ab-energy} from below, we shall first prove 
that the pair of polynomials $\alpha$ and $\mu,$ with  $\alpha$ given in  \eqref{pol:alpha-beta} for $\vartheta=3$,
\begin{equation}\label{1.8}
\begin{split}
\alpha (\zeta)=\frac{1}{420}\big ( 3147\zeta^7&-10038\zeta^6+15750\zeta^5-16800\zeta^4\\
&+12425\zeta^3-5922\zeta^2 +1638\zeta-200\big),
\end{split}
\end{equation}
and $\mu$ the polynomial associated to the multipliers in \eqref{mu2},
\begin{equation}\label{mu1}
\mu(\zeta)= \zeta^7-1.6\zeta^6+1.6\zeta^5-1.1\zeta^4+0.3\zeta^3, 
\end{equation}
satisfy the conditions in Definition \ref{De:multipliers}; see Proposition \ref{Pr:alpha-mu}.
This fact in  combination with Lemma \ref{Le:Dahl} will enable us to utilize a relation
of the form  \eqref{G}.

Analogously, to estimate $A_{n+7}$ from below, in view of its specific form and, in particular, the fact that
it depends only on five consecutive approximations, namely on $u^{n+3},\dotsc,u^{n+7},$ 
it suffices to use polynomials of degree $4.$ Thus, we factor $\zeta^3$  out of the polynomials $\beta$ 
in \eqref{pol:alpha-beta} for $\vartheta=3$ and $\mu$ and consider the polynomials $\delta$ and $\kappa,$
\begin{equation}\label{po:delta}
\delta(\zeta):=3\zeta^4-2\zeta^3,
\end{equation}
and
\begin{equation}\label{po:kappa}
\kappa(\zeta):= \zeta^4-\frac{8}{5}\zeta^3+\frac{8}{5}\zeta^2-\frac{11}{10}\zeta+\frac{3}{10}
=\big(\zeta-\frac{3}{5}\big) \big(\zeta^3-\zeta^2+\zeta-\frac{1}{2}\big);
\end{equation}
cf.\ \eqref{A-beta-mu-delta-kappa}. Now, to take advantage of a relation of the form  \eqref{G}, given that the polynomials
$\delta-c_{\star}\kappa$ and $\kappa$ enter into the first and second arguments
in the inner product in $A_{n+7},$ we need to prove that the pair of polynomial
$(\delta-c_{\star}\kappa,\kappa)$ satisfies the conditions in Lemma \ref{Le:Dahl};
obviously,  these conditions can be reformulated in the form that 
the polynomials $\delta$ and $\kappa$  have no common divisor and satisfy
condition \eqref{A-delta-kappa}. We shall prove these properties in Proposition \ref{Pr:delta-kappa}.

\begin{proposition}[Polynomials $\alpha$ and $\mu$ satisfy condition \eqref{A-alpha-mu}]\label{Pr:alpha-mu}
The polynomials $\alpha$ for $\vartheta=3$ and $\mu$ of \eqref{1.8} and \eqref{mu1}  do not have common divisor
and satisfy condition \eqref{A-alpha-mu}.
\end{proposition}

\begin{proof}
First, $\mu(\zeta)=\zeta^3(\zeta-3/5)\tilde \kappa(\zeta)$ with $\tilde \kappa(\zeta):=\zeta^3-\zeta^2+\zeta-1/2;$
see \eqref{mu1} and  \eqref{po:kappa}. Thus, to show that the roots of $\mu$ are inside the unit disk,
it suffices to show that this is the case for $\tilde \kappa.$
Now,
\begin{equation*}
\tilde \kappa\big (\frac 12\big )=-\frac{1}{8}<0\quad\text{and}\quad \tilde \kappa(1)=\frac{1}{2}>0,
\end{equation*}
and thus $\tilde \kappa$ has a real root $\zeta_1\in (1/2,1).$ Actually, this is the only real root of $\tilde \kappa,$
since $\tilde \kappa$ is strictly increasing on the real axis,
\begin{equation*}
\tilde \kappa'(x)=3x^2-2x+1=2x^2+(x-1)^2> 0.
\end{equation*}
Let  $\zeta_2,\zeta_3$ be the complex conjugate roots of $\tilde \kappa.$
Then, according to Vieta's formulas,
\begin{equation*}
\zeta_1\zeta_2\zeta_3=\zeta_1|\zeta_2|^2=\frac{1}{2},
\end{equation*}
which, in combination with $\zeta_1>1/2,$ implies $|\zeta_2|<1.$
Thus, $|\zeta_1|,|\zeta_2|,|\zeta_3|<1$. We infer that all roots of $\mu$ are inside the unit disk.

The generating polynomial $\alpha$ of the WSBDF7 method \eqref{ab-3} is
\begin{equation*}
\label{alpha1}
420 \alpha(\zeta)=3147\zeta^7-10038\zeta^6+15750\zeta^5-16800\zeta^4+12425\zeta^3-5922\zeta^2+1638\zeta-200;
\end{equation*}
see \eqref{1.8}.
First, $\alpha(0)=-10/21, $ $\alpha(3/5)=53/11822. $
Furthermore, $420\alpha(\zeta)=(\zeta-1)\tilde \alpha(\zeta)$ with
\begin{equation*}
\begin{aligned}
\tilde\alpha(\zeta):={}&3147\zeta^6-6891\zeta^5+8859\zeta^4-7941\zeta^3+4484\zeta^2-1438\zeta+200\\
={}&\Big (3147\zeta^3-3744\zeta^2+1968\zeta-\frac{1311}{2}\Big)\tilde \kappa(\zeta)-\frac{1}{4}\nu (\zeta)
\end{aligned}
\end{equation*}
and $\nu (\zeta):=46\zeta^2-806\zeta+511,$ and it is easy to check that none of the roots of the quadratic polynomial
$\nu$ is a root of $ \tilde \kappa;$ consequently, the polynomials $\tilde \alpha$ and $\tilde \kappa$ do not have common divisor.
We then easily infer that the polynomials $\alpha$ and $\mu$ do not have common divisor. 

Now, the rational function $\alpha/\mu$ is holomorphic outside the unit disk in the
complex plane and
\[\lim_{|z|\to \infty}\frac {\alpha(z)}{\mu (z)}=\alpha_7=\frac {1049}{140}>0.\]
Therefore, according to the maximum principle for harmonic functions,
the A-stability property \eqref{A} is equivalent to
\begin{equation*}
\Real \frac {\alpha(\zeta)}{\mu(\zeta)}\geqslant 0 \quad \forall \zeta\in \K,
\end{equation*}
with $\K$ the unit circle in the complex plane, i.e., equivalent to
\begin{equation*}
\label{Real1}
\Real \big [\alpha(\e^{\i \varphi})\mu (\e^{-\i \varphi})\big ]\geqslant 0 \quad \forall \varphi \in \Re.
\end{equation*}
In view of \eqref{mu1}, this property  takes the form
\begin{equation}
\label{Real2}
\Real \big [420\alpha(\e^{\i \varphi})\e^{-\i 3\varphi} \kappa(\e^{-\i \varphi})\big ]
\geqslant 0 \quad \forall \varphi \in \Re.
\end{equation}

Now, it is easily seen that
\begin{equation*}
\begin{aligned}
420\alpha(\e^{\i \varphi})\e^{-\i 3\varphi}
={}&3147\cos(4\varphi)-10238\cos(3\varphi)+17388\cos(2\varphi)-22722\cos\varphi+12425\\
&+\i\big [3147\sin(4\varphi)-9838\sin(3\varphi)+14112\sin(2\varphi)-10878\sin\varphi\big ].
\end{aligned}
\end{equation*}
With $x:=\cos\varphi,$ recalling the elementary trigonometric identities
\begin{equation*}
\begin{alignedat}{3}
\cos(2\varphi) &=2x^2-1,\quad  &&\cos(3\varphi) =4x^3-3x, \quad&&\cos(4\varphi) =8x^4-8x^2+1,\\
\sin(2\varphi) &=2x\sin\varphi, \quad&&\sin(3\varphi) =(4x^2-1)\sin\varphi,\quad&&\sin(4\varphi) =(8x^3-4x)\sin\varphi,
\end{alignedat}
\end{equation*}
we see that
\begin{equation}\label{Real3}
\begin{aligned}
420\alpha(\e^{\i \varphi})\e^{-\i 3\varphi}={}&8(1-x)(-3147x^3+1972x^2+772x-227)\\
&+\i4(6294x^3 - 9838x^2 + 3909x - 260)\sin\varphi.
\end{aligned}
\end{equation}
Notice that the factor $1-x$ in the real part of $\alpha(\e^{\i \varphi})\e^{-\i 3\varphi}$
is due to the fact that $\alpha(1)=0.$ Similarly,
\[\begin{aligned}
\kappa(\e^{-\i \varphi})={}&\cos(4\varphi)-1.6\cos(3\varphi)+1.6\cos(2\varphi)-1.1\cos\varphi+0.3\\
&-\i\big [\sin(4\varphi)-1.6\sin(3\varphi)+1.6\sin(2\varphi)-1.1\sin\varphi\big ]
\end{aligned}\]
and
\begin{equation}
\label{Real4}
\kappa(\e^{-\i \varphi})=\frac{1}{10}(80x^4 - 64x^3 - 48x^2 + 37x - 3)-\i \frac{1}{10}(80x^3 - 64x^2 - 8x + 5)\sin\varphi.
\end{equation}
In view of \eqref{Real3} and \eqref{Real4}, the desired property  \eqref{Real2}
can be written in the form
\begin{equation*}
\label{Real5}
\frac{2}{5}(1-x)P(x)\geqslant 0 \quad \forall x \in [-1,1]
\end{equation*}
with
\[\begin{aligned}
P(x)&:=2(-3147x^3+1972x^2+772x-227)(80x^4 - 64x^3 - 48x^2 + 37x - 3)\\
&\quad+(1+x)(6294x^3 - 9838x^2 + 3909x - 260)(80x^3 - 64x^2 - 8x + 5),
\end{aligned}\]
i.e.,
\begin{equation}
\label{Real6}
P(x)=32000x^6 - 124640x^5 + 173872x^4 - 104870x^3 + 24891x^2 - 1105x + 62.
\end{equation}
Now, $P$ is positive in the interval $[-1,1],$ and thus  \eqref{Real2} is valid.
Indeed, first, all terms are positive for $-1\leqslant x< 0$,
whence $P$ is positive in $[-1,0)$. For $0\leqslant x\leqslant 1$,
see  Figure \ref{Fig:P}.
\end{proof}
\begin{figure}[!ht]
\centering
\psset{yunit=0.012cm,xunit=4.5cm}
\begin{pspicture}(-0.25,-32.1)(1.25,305)
\psaxes[ticks=none,labels=none,linewidth=0.6pt]{->}(0,0)(-0.25,-31.8)(1.2,288)%
[$\!\!x$,0][$ $,90]
\psset{plotpoints=10000}
\psPolynomial[coeff=62 -1105 24891 -104870 173872 -124640 32000, linewidth=0.5pt,linecolor=blue]{0}{1}
\uput[0](0.3,160){\small $P$}
\uput[0](-0.063,300.7){\small $y$}
\uput[0](0.945,-20){\small $1$}
\pscircle*(0,100){0.038}
\pscircle*(0,200){0.038}
\uput[0](-0.19,100){\small $100$}
\uput[0](-0.19,200){\small $200$}
\uput[0](-0.12,-20.6){\small $O$}
\pscircle*(1,0){0.038}
\end{pspicture}
\caption{The graph of  polynomial $P$ of \eqref{Real6} in the interval $[0,1]$.}
\label{Fig:P}
\end{figure}
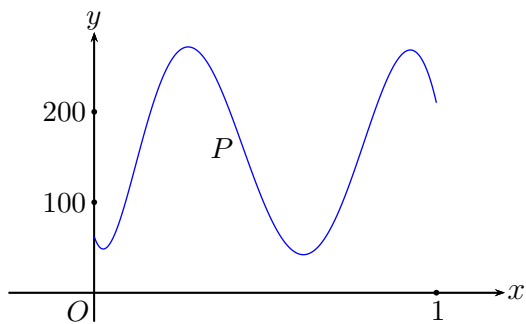

\begin{proposition}[Polynomials $\delta$ and $\kappa$ satisfy condition \eqref{A-delta-kappa}]\label{Pr:delta-kappa}
The polynomials $\delta$ and $\kappa$ of \eqref{po:delta} and \eqref{po:kappa} do not have common divisor
and satisfy the analogue of  condition \eqref{trig-ineq2} with $\tilde c_{\star}=0.01,$ and thus condition
\eqref{A-delta-kappa} for some positive constant $c_{\star}.$
In addition, the dual to \eqref{A-delta-kappa} is also valid,
\begin{equation}
\label{A-kappa-delta}
\Real \frac {\kappa(\zeta)}{\delta(\zeta)}>\hat c_\star\quad\text{for }\, |\zeta|>1,
\end{equation}
for some positive constant $\hat c_\star.$
\end{proposition}

\begin{proof}
First, $\kappa(0)=3/10$ and $\kappa(2/3)=1/810,$ whence the polynomials $\delta$ and $\kappa$ have no common divisor.

Since the roots of  $\kappa$ are inside  the unit disk, the rational function $\delta/\kappa$ is holomorphic outside the unit disk in the
complex plane; see the proof of Proposition \ref{Pr:alpha-mu}. Furthermore,
\[\lim_{|z|\to \infty}\frac {\delta(z)}{\kappa(z)}=3\geqslant c_{\star}.\]
Therefore, according to the maximum principle for harmonic functions,
the A-stability property \eqref{A-delta-kappa} is equivalent to
\begin{equation*}
\Real \frac {\delta(\zeta)}{\kappa(\zeta)}\geqslant c_{\star} \quad \forall \zeta\in \K, 
\end{equation*}
that is, equivalent to
\begin{equation*}\label{1.19-old}
\Real \big [(3\e^{\i 2\varphi}-2\e^{\i \varphi})\e^{\i 2\varphi}\kappa(\e^{-\i \varphi})\big ]\geqslant c_{\star} |\kappa(\e^{\i \varphi})|^2\quad \forall \varphi \in \Re.
\end{equation*}
Thus, it suffices to show that
\begin{equation}\label{1.19}
\Real \big [(3\e^{\i 2\varphi}-2\e^{\i \varphi})\e^{\i 2\varphi}\kappa(\e^{-\i \varphi})\big ]\geqslant \tilde c_{\star}=0.01\quad \forall \varphi \in \Re.
\end{equation}

With $x:=\cos\varphi,$ recalling the elementary trigonometric identities
\begin{equation*}
\cos(2\varphi) =2x^2-1,\quad \sin(2\varphi) =2x\sin\varphi,
\end{equation*}
we easily see that
\begin{equation}
\label{1.20}
\begin{aligned}
3\e^{\i 2\varphi}-2\e^{\i \varphi}
&=\big (3\cos(2\varphi)-2\cos\varphi\big )
+\i\big (3\sin(2\varphi)-2\sin\varphi\big )\\
&=(6x^2-2x-3)+\i2(3x-1)\sin\varphi.
\end{aligned}
\end{equation}
Similarly,
\begin{equation}
\label{1.21}
\begin{aligned}
\e^{\i 2\varphi}\kappa(\e^{-\i \varphi})
&=\big (1.3\cos(2\varphi)-2.7\cos\varphi+1.6\big )
-\i\big (0.7\sin(2\varphi)-0.5\sin\varphi\big )\\
&=(2.6x^2-2.7x+ 0.3)-\i (1.4x-0.5)\sin\varphi.
\end{aligned}
\end{equation}
In view of \eqref{1.20} and \eqref{1.21}, the desired property  \eqref{1.19}
can be written in the form 
\begin{equation}
\label{1.22}
\begin{aligned}
g(x):={}&(6x^2-2x-3)(2.6x^2-2.7x+ 0.3)+2(1-x^2)(3x-1)(1.4x-0.5)\\
={}&7.2x^4 - 15.6x^3 + 6.8x^2 + 1.7x + 0.1\geqslant \tilde c_{\star}, \quad x\in [-1,1].
\end{aligned}
\end{equation}
Now, $g$ attains its minimum $0.01379862357$ in $[-1,1]$ at $-0.09331476$;
this value of the minimum is the motivation for choosing $\tilde c_{\star}=0.01$ in \eqref{trig-ineq2} 
for the multiplier $(1.6,-1.6,1.1,-0.3,$ $0,0,0)$. 
Thus, \eqref{1.19} is valid. See also Figure \ref{Fig:PP}.

Let us provide also a complete theoretical proof of \eqref{1.22}.
For negative and positive $x,$ we write $g$ in the form $g(x)=7.2x^4 - 15.6x^3 + 0.01 + g_1(x)=
x^2g_2(x)+ 1.7x + 0.1,$ with
\[g_1(x)=6.8x^2 + 1.7x + 0.09\quad\text{and}\quad g_2(x)=7.2x^2 - 15.6x + 6.8.\]
The roots $x_{1,2}$ and $x_{3,4}$ of $g_1$ and $g_2,$ respectively, are
\[x_1=-0.17388461,\quad x_2=-0.07611538,\quad x_3=0.60461977,\quad x_4=1.56204688.\]
Therefore, $g_1$ and $g_2,$ respectively, are positive outside the intervals $[x_1,x_2]$
and $[x_3,x_4],$ and we easily see that $g(x)\geqslant 0.01$ in $[-1,-0.17388461]$
and in  $[-0.07611538,0.60461977].$ 

Furthermore, 
\[g'(x)=28.8x^3-46.8x^2+13.6x+1.7\quad\text{and}\quad
g''(x)=86.4x^2-93.6x+13.6.\]
The roots of $g''$ are $x_5=0.17289116$ and $x_6=0.91044216$. Therefore, $g''$ is negative in
$[0.60461977,x_6]$ and  positive in $[x_6,1],$ whence $g'$ is decreasing in
$[0.60461977,x_6]$ and  increasing  in $[x_6,1].$ Since $g'(0.60461977)=-0.82001368$ and  
$g'(1)=-2.7,$ we see that $g'$ is negative in $[0.60461977,1],$ whence  $g$ is decreasing in
$[0.60461977,1].$ Therefore, $g(x)\geqslant g(1)=0.2$ for $x\in [0.60461977,1].$
Summarizing, up to now, we proved that  \eqref{1.22} is valid in $[-1,-0.17388461]$
and in  $[-0.07611538,1].$ 

Finally, let us write $g$ in the form $g(x)=x^2g_3(x)+g_4(x)+0.01$ with 
\[g_3(x)=7.2x^2 - 15.6x -1.228
\quad\text{and}\quad g_4(x)=8.028x^2 +1.7x + 0.09.\]
The function $g_3$ is obviously decreasing in $[-1,-0.07611538].$ 
Since $g_3(-0.07611538)= 0.00111349,$ and the discriminant of $g_4$ is $-8\cdot10^{-5},$ we see that $g_3$ and $g_4$ are positive in $[-1,-0.07611538].$
Therefore, $g(x)>0.01$ for $x\in [-1,-0.07611538].$  This completes the proof of \eqref{1.22}.

The roots of  $\delta$ are $\zeta_1=0$ and $\zeta_2=2/3,$ whence the rational function $\kappa/\delta$ is holomorphic 
outside the unit disk in the complex plane. Therefore, \eqref{A-kappa-delta}, for $\hat c_\star\leqslant 1/3,$
is equivalent to
\begin{equation*}
\Real \frac {\kappa(\bar\zeta)}{\delta(\bar\zeta)}\geqslant \hat c_{\star} \quad \forall \zeta\in \K, 
\end{equation*}
that is, equivalent to
\begin{equation*}\label{1.19-old-dual}
\Real \big [(3\e^{\i 2\varphi}-2\e^{\i \varphi})\e^{\i 2\varphi}\kappa(\e^{-\i \varphi})\big ]\geqslant \hat c_{\star} |\delta(\e^{\i \varphi})|^2\quad \forall \varphi \in \Re.
\end{equation*}
Obviously, $ |\delta(\e^{\i \varphi})|\leqslant 5.$  We have already seen in \eqref{1.19} that the function on the left-hand 
side is strictly positive, and easily infer that this inequality is indeed valid for some positive constant $\hat c_\star.$
\end{proof}

\begin{figure}[!ht]
\centering
%
\psset{yunit=1.5cm,xunit=4.5cm}
\begin{pspicture}(-0.45,-0.301)(1.25,1.95)
\psaxes[ticks=none,labels=none,linewidth=0.6pt]{->}(0,0)(-0.43,-0.3)(1.2,1.8)%
[$\!\!x$,0][$ $,90]
\psset{plotpoints=10000}
\psPolynomial[coeff= 0.09 1.7  6.8  -15.6  7.2  , linewidth=0.5pt,linecolor=blue]{-0.4}{1}
\uput[0](0.3,0.87){\small $g-\tilde c_\star$}
\uput[0](-0.063,1.9){\small $y$}
\uput[0](0.945,-0.15){\small $1$}
\pscircle*(0,1){0.038}
\uput[0](-0.094,1){\small $1$}
\uput[0](-0.12,-0.15){\small $O$}
\pscircle*(1,0){0.038}
\end{pspicture}
\quad\quad
\psset{yunit=15.5cm,xunit=4.5cm}
\begin{pspicture}(-0.22,-0.03)(0.25,0.141)
\psaxes[ticks=none,labels=none,linewidth=0.6pt]{->}(0,0)(-0.22,-0.023)(0.2,0.118)%
[$\!\!x$,0][$ $,90]
\psset{plotpoints=10000}
\psPolynomial[coeff= 0.09 1.7  6.8  -15.6  7.2  , linewidth=0.5pt,linecolor=blue]{-0.17}{0}
\uput[0](-0.063,0.13){\small $y$}
\pscircle*(-0.0933,0){0.038}
\uput[0](-0.28,-0.015){\tiny $-0.0933$}
\end{pspicture}
\caption{The graph of the polynomial $g-\tilde c_\star$ of \eqref{1.22} in the interval $[-0.4,1]$, left,
and zoom  in in the interval $[-0.17,0],$ right.}
\label{Fig:PP}
\end{figure}
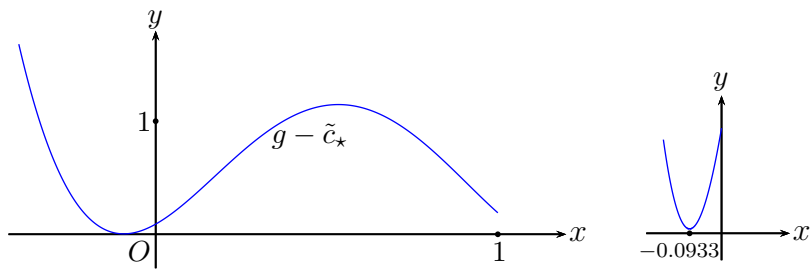

\section{On the determination of multipliers}\label{Se:det-mult}
In this section, the objective is the derivation of necessary conditions for multipliers 
for the seven-step WSBDF method with $\vartheta=3$ such that $\mu_5=\mu_6=\mu_7=0;$
we utilized these conditions to determine the multipliers \eqref{mu2}.
Let us mention that multipliers with $\mu_4=\mu_5=\mu_6=\mu_7=0$  do not exist;
see Remark \ref{Re:mu4}.

In the case  $\mu_5=\mu_6=\mu_7=0,$ we have $\mu(\zeta)=\zeta^3\kappa(\zeta)$ with $\kappa(\zeta)=\zeta^4-\mu_1\zeta^3
-\mu_2\zeta^2-\mu_3\zeta-\mu_4$, and, provided that the roots of $\kappa$ lie in the unit disk,
the A-stability condition \eqref{A-alpha-mu} takes the form
\begin{equation*}
4(1-x)P(x)\geqslant 0 \quad \forall x \in [-1,1]
\end{equation*}
with
\[\begin{aligned}
P(x):={}&2(-3147x^3+1972x^2+772x-227)\cdot\\
&\cdot\big(8x^4-8x^2+1-\mu_1(4x^3-3x)-\mu_2(2x^2-1)-\mu_3x-\mu_4\big)\\
&+(1+x)(6294x^3 - 9838x^2 + 3909x - 260)\big(8x^3-4x -\mu_1(4x^2-1)- 2\mu_2x -\mu_3\big),
\end{aligned}\]
i.e.,
\begin{equation}
\label{A1}
\begin{split}
P(x)={}&3200x^6 - 9904x^5 + 8184x^4 +2990x^3 -7020x^2 +2584x -454\\
&+(-1600x^5 + 4952x^4 -4492x^3 -257x^2 +2287x -260)\mu_1\\
&+(-800x^4 +2476x^3 -2446x^2 +2064x -454)\mu_2\\
&+(-400x^3 +4385x^2 -3195x +260)\mu_3\\
&+(6294x^3 -3944x^2 -1544x +454)\mu_4\geqslant 0 \quad \forall x \in [-1,1].
\end{split}
\end{equation}

Analogously,  condition \eqref{A-delta-kappa}, for some positive constant $c_\star,$
 leads to the strict inequality condition
\begin{equation*}
g(x)> 0 \quad \forall x \in [-1,1]
\end{equation*}
with
\[\begin{aligned}
g(x):={}&\big(3(8x^4-8x^2+1)-2(4x^3-3x)\big)\cdot \\
&\cdot\big(8x^4-8x^2+1-\mu_1(4x^3-3x)-\mu_2(2x^2-1)-\mu_3x-\mu_4\big)\\
&+(1-x^2)\big(3(8x^3-4x)-2(4x^2-1)\big)\big(8x^3-4x -\mu_1(4x^2-1)- 2\mu_2x -\mu_3\big),
\end{aligned}\]
i.e., to
\begin{equation}\label{A2}
\begin{split}
g(x)=&{}-2x+3+(-3x+2)\mu_1+(-6x^2 +2x +3)\mu_2\\
&+(-12x^3 +4x^2 +9x -2)\mu_3+(-24x^4+8x^3 +24x^2 -6x-3)\mu_4> 0
\end{split}
\end{equation}
for all $x\in  [-1,1].$ Notice that the strict inequality in \eqref{A2} implies that
\eqref{trig-ineq2} is satisfied with $\tilde c_\star:=\min_{-1\leqslant x\leqslant 1}g(x);$
consequently, \eqref{A-delta-kappa} is satisfied for some positive constant $c_\star.$

Necessary conditions for \eqref{A1} and \eqref{A2} can be derived by evaluating $P$ and $g$ at certain points. 
For instance, we obtain the following necessary condition, which we utilized to determine the multipliers \eqref{mu2}. 

\begin{lemma}[Range of multipliers]\label{Le:constr-mult}
If $(\mu_1,\mu_2,\mu_3,\mu_4,0,0,0)$ is a multiplier of the seven-step WSBDF method with $\vartheta=3$, 
then there holds
\begin{equation*}\label{eq:range-multipl}
\begin{alignedat}{2}
&1.5561\leqslant\mu_1< 2.3133,\quad &&-2.2024<\mu_2<-1.4259,\\
&0.5394<\mu_3< 1.3955,\quad &&-0.6518<\mu_4< -0.0504.
\end{alignedat}
\end{equation*}
\end{lemma}

\begin{proof}
First,
\begin{equation*}
\begin{alignedat}{2}
&g(1)=1-\mu_1-\mu_2-\mu_3-\mu_4, &&2g(-1/2)=8+7\mu_1+\mu_2-8\mu_3+7\mu_4,\\
&g(0)=2(\mu_1-\mu_3)-3(\mu_4-\mu_2-1).\quad &&\\
\end{alignedat}
\end{equation*}

Furthermore, 
\begin{equation}\label{eq:mult-ness-cond1}
g(0)+2g(1)=5+\mu_2-4\mu_3-5\mu_4> 0,
\end{equation}
\begin{equation}\label{eq:mult-ness-cond5}
7g(1)+2g(-1/2)=15-6\mu_2-15\mu_3> 0.
\end{equation}

Also,
\begin{equation*}
10^{-3}P(0.999999)/0.21=-2+3\mu_1+4\mu_2+5\mu_3+6\mu_4,
\end{equation*}
whence
\begin{equation}\label{eq:mult-ness-cond3}
3g(1)+10^{-3}P(0.999999)/0.21=1+\mu_2+2\mu_3+3\mu_4> 0.
\end{equation}

Multiplying \eqref{eq:mult-ness-cond3} by $2$ and adding the result to \eqref{eq:mult-ness-cond1},
we get $7+3\mu_2+\mu_4> 0.$ 
Adding the positive quantities
\begin{equation*}
\begin{split}
g(0.214929)/1.3552&=1.8965+\mu_1+2.3264\mu_2-2.3264\mu_4,\\
g(0.941785)/0.8254&=1.3526-\mu_1-0.5309\mu_2+0.5309\mu_4,
\end{split}    
\end{equation*}
we obtain $3.2491+1.7955\mu_2-1.7955\mu_4> 0,$ 
which, in combination with $7+3\mu_2+\mu_4>0,$ yields
$15.8176+7.1820\mu_2>0,$  i.e., $\mu_2>-2.2024.$
Similarly, adding the positive quantities
\begin{equation*}
\begin{split}
g(0.214929)/3.1527&=0.8152+0.4299\mu_1+\mu_2-\mu_4,\\
g(0.941785)/0.4382&=2.5477-1.8836\mu_1-\mu_2+\mu_4,
\end{split}    
\end{equation*}
we obtain $3.3629-1.4537\mu_1>0,$  i.e. $\mu_1<2.3133.$

Summation of the nonnegative  quantities
\begin{equation*}
\begin{split}
P(0.68481813)/431.5990&=-0.8759+1.3696\mu_1+\mu_2-\mu_4,\\
P(0.09319637)/280.9437&=-0.9653-0.1864\mu_1-\mu_2+\mu_4,
\end{split}    
\end{equation*} 
leads to $-1.8412+1.1832\mu_1\geqslant0,$ i.e., $\mu_1\geqslant 1.5561.$

Adding the positive quantities
\begin{equation*}
\begin{split}
g(-0.264464)/2.7934&=1.2633+\mu_1+0.7344\mu_2-1.3884\mu_3,\\
g(0.962266)/0.8868&=1.2128-\mu_1-0.7118\mu_2-0.3699\mu_3,
\end{split}    
\end{equation*} 
we obtain $2.4761+0.0226\mu_2-1.7583\mu_3>0,$ which, in combination with \eqref{eq:mult-ness-cond5},
yields $15.1956-10.8888\mu_3>0,$ i.e. $\mu_3< 1.3955.$

Also, summation of the nonnegative  quantities
\begin{equation*}
\begin{split}
P(-1/9)/510.3402&=-1.6273-\mu_1-1.4050\mu_2+1.3122\mu_3+1.1134\mu_4,\\
P(1/2)/517.2500&=-0.5631+\mu_1+0.4369\mu_2-0.5631\mu_3-\mu_4,
\end{split}
\end{equation*}  
yields $-2.1904-0.9681\mu_2+0.7491\mu_3+0.1134\mu_4\geqslant0,$ 
which, together with \eqref{eq:mult-ness-cond1}, yields
$-5.0152-3.1232\mu_2-3.2924\mu_4>0.$
Similarly, adding the nonnegative quantities
\begin{equation*}
\begin{split}
P(0.68481813)/591.1336&=-0.6395+\mu_1+0.7301\mu_2-0.7301\mu_4,\\
P(0.09319637)/52.3659&=-5.1786-\mu_1-5.3650\mu_2+5.3650\mu_4,
\end{split}    
\end{equation*} 
we have $-5.8181-4.6349\mu_2+4.6349\mu_4\geqslant 0,$
which, together with $-5.0152-3.1232\mu_2-3.2924\mu_4>0,$
yields $-42.4005-29.7357\mu_2>0$, i.e., $\mu_2<-1.4259.$

From the positivity of $g(1)$  and
\begin{equation*}
g(-0.26)/2.7800=1.2662+\mu_1+0.7462\mu_2-1.3880\mu_3-0.0244\mu_4>0,  
\end{equation*} 
we have $2.2662-0.2538\mu_2-2.3880\mu_3-1.0244\mu_4>0,$
which, together with \eqref{eq:mult-ness-cond3},
yields $6.9204+1.8804\mu_2+5.1152\mu_4>0$. Combining the latter condition with 
$-5.8181-4.6349\mu_2+4.6349\mu_4\geqslant 0,$ we obtain $21.1350+32.4239\mu_4>0$
 i.e., $\mu_4>-0.6518.$

Adding the nonnegative quantities
\begin{equation*}
\begin{split}
P(0.81865385)/611.3646&=-0.6632+\mu_1+0.9741\mu_2+0.5950\mu_3,\\
10^{-3}P(-0.407988)/0.7755&=-3.2576-\mu_1-2.4417\mu_2+2.9924\mu_3,
\end{split}    
\end{equation*}  
we get $-3.9208-1.4676\mu_2+3.5874\mu_3\geqslant 0.$
Similarly, adding the nonnegative and positive, respectively, quantities
\begin{equation*}
\begin{split}
P(-1/9)/510.3402&=-1.6273-\mu_1-1.4050\mu_2+1.3122\mu_3+1.1134\mu_4,\\
g(1/33)/1.9091&=1.5397+\mu_1+1.6003\mu_2-0.9030\mu_3-1.6550\mu_4,
\end{split}
\end{equation*}
we have $-0.0876+0.1952\mu_2+0.4092\mu_3-0.5416\mu_4>0,$
which, together with \eqref{eq:mult-ness-cond3},
yields $0.2788+1.1272\mu_2+2.3108\mu_3>0.$
Combining the latter relation with the already established relation $-3.9208-1.4676\mu_2+3.5874\mu_3\geqslant 0$, 
we get $-4.0104+7.4350\mu_3> 0$, i.e., $\mu_3>0.5394.$

Combining
$-2.1904-0.9681\mu_2+0.7491\mu_3+0.1134\mu_4\geqslant0$   with
$-0.0876+0.1952\mu_2+0.4092\mu_3-0.5416\mu_4>0,$
we obtain $-0.5124+0.5424\mu_3-0.5022\mu_4>0.$
From the positivity of $g(1)$ and the nonnegativity of 
\begin{equation*}
P(0.7)/594.2772=-0.6437+\mu_1+0.7563\mu_2+0.0588\mu_3-0.6740\mu_4,
\end{equation*} 
we get $0.3563-0.2437\mu_2-0.9412\mu_3-1.6740\mu_4>0$.
Combined with \eqref{eq:mult-ness-cond1},
the latter relation yields
$1.5748-1.9160\mu_3-2.8925\mu_4>0,$ 
which together with $-0.5124+0.5424\mu_3-0.5022\mu_4>0$
leads to $-0.1276-2.5311\mu_4>0,$ i.e.,  $\mu_4<-0.0504.$
The proof is complete.
\end{proof}

\begin{remark}[Nonexistence of multipliers of the form $(\mu_1,\mu_2,\mu_3,0,0,0,0)$]\label{Re:mu4}
Our first attempt was to determine multipliers of the form  $(\mu_1,\mu_2,\mu_3,0,0,0,0).$
That such multipliers do not exist follows immediately from Lemma \ref{Le:constr-mult}
since $\mu_4$ must be negative.
\end{remark}

\section{Stability}\label{Se:stab}
Here we prove the stability estimates, Theorems \ref{Th:stab} and \ref{Th:stab2}, and Corollary \ref{Co:stab}.

\subsection{Proof of Theorem \ref{Th:stab}}
According to Propositions \ref{Pr:alpha-mu} and \ref{Pr:delta-kappa}, respectively, in combination with Lemma \ref{Le:Dahl},
there exist two positive definite symmetric matrices $G=(g_{ij})\in \Re^{7,7}$ and $\widetilde G=(\tilde g_{ij})\in \Re^{4,4}$
such that with the notation $\mathcal{U}^n:=(u^{n-6},\dotsc,u^{n})^\top$
and $\bm{U}^n:=(u^{n-3},\dotsc,u^{n})^\top,$ and the norms $|\mathcal{U}^n|_G$ and $\|\bm{U}^n\|_{\widetilde G}$ given by
 \begin{equation}\label{G-norms}
 |\mathcal{U}^n|_G^2=\sum_{i,j=1}^7g_{ij}(u^{n-7+i},u^{n-7+j}),\quad  \|\bm{U}^n\|_{\widetilde G}^2=\sum_{i,j=1}^4{\tilde g}_{ij}\langle u^{n-4+i},u^{n-4+j}\rangle,
\end{equation}
there holds
\begin{equation}\label{G-stab}
\Big (\sum\limits^7_{i=0}\alpha_i  u^{n+i},u^{n+7}-\sum_{j=1}^4\mu_j u^{n+7-j}\Big )
\geqslant |\mathcal{U}^{n+7}|_G^2-|\mathcal{U}^{n+6}|_G^2
\end{equation}
and
\begin{equation}\label{2.7}
A_{n+7}\geqslant \|\bm{U}^{n+7}\|_{\widetilde G}^2-\|\bm{U}^{n+6}\|_{\widetilde G}^2.
\end{equation}

Utilizing \eqref{G-stab} and  \eqref{2.7}, we infer from \eqref{ab-energy} that
\begin{equation}\label{2.8}
|\mathcal{U}^{n+7}|_G^2-|\mathcal{U}^{n+6}|_G^2+\tau\|\bm{U}^{n+7}\|_{\widetilde G}^2-\tau\|\bm{U}^{n+6}\|_{\widetilde G}^2
+c_{\star}\tau\Big\|u^{n+7}-\sum_{j=1}^4\mu_j u^{n+7-j}\Big\|^2\leqslant \tau F_{n+7}.
\end{equation}

Furthermore, the terms involving the forcing term can be estimated
by elementary inequalities in the form
\begin{equation*}
\begin{aligned}
F_{n+7}&\leqslant\left\|3 f^{n+7}-2f^{n+6}\right\|_\star\Big\|u^{n+7}-\sum_{j=1}^4\mu_j u^{n+7-j}\Big\|\\
&\leqslant \frac 1{2c_{\star}}\|3 f^{n+7}-2f^{n+6}\|_\star^2+
\frac{c_{\star}}{2}\Big\|u^{n+7}-\sum_{j=1}^4\mu_j u^{n+7-j}\Big\|^2.
\end{aligned}
\end{equation*}
Using this estimate in \eqref{2.8} and summing  over $n,$ from $n=0$ to $n=m-7$,  we obtain
\begin{equation}\label{2.9}
\begin{aligned}
 |\mathcal{U}^{m}|_G^2-|\mathcal{U}^{6}|_G^2&+\tau\|\bm{U}^{m}\|_{\widetilde G}^2-\tau\|\bm{U}^{6}\|_{\widetilde G}^2
 +\frac{c_{\star}\tau}{2}\sum_{n=7}^m\Big\|u^{n}-\sum_{j=1}^4\mu_j u^{n-j}\Big\|^2\\
 &\leqslant  \frac \tau{2c_{\star}}\sum_{n=7}^m\|3 f^{n}-2f^{n-1}\|_\star^2.
\end{aligned}
\end{equation}

Now, with $c_1$ and $c_2$ the smallest eigenvalues of the matrices $G$ and $\widetilde G$, respectively, we have
\begin{equation*}
|\mathcal{U}^{m}|_G^2\geqslant c_1|u^m|^2,~~ \|\bm{U}^{m}\|_{\widetilde G}^2\geqslant c_2\|u^m\|^2, ~~ |\mathcal{U}^{6}|_G^2\leqslant C\sum_{j=0}^6 |u^j|^2,
~~ \|\bm{U}^{6}\|_{\widetilde G}^2\leqslant C\sum_{j=0}^6 \|u^j\|^2.
 \end{equation*}
Thus, \eqref{2.9} yields, ignoring the last term on its heft-hand side,
\begin{equation*}
|u^m|^2+\tau\|u^m\|^2\leqslant C\sum_{j=0}^6 \left(|u^j|^2+\tau\|u^j\|^2\right)+C\tau\sum_{n=7}^m\|3 f^{n}-2f^{n-1}\|_\star^2.
\end{equation*}
The desired result \eqref{stab-abg2} is an obvious consequence of this estimate.

\subsection{Proof of Theorem \ref{Th:stab2}}
With the notation \eqref{BDF-oper}, we write the WSBDF7 method \eqref{ab-3}  in the form
\begin{equation}
\label{ab-new}
\dot u^n+ 3 A u^n-2A u^{n-1}= 3 f^n-2f^{n-1},\quad n=7,\dotsc,N.
\end{equation}

Let us introduce the notation
\begin{equation}
\label{notation-n}
v^m:=3u^m-2u^{m-1}\quad\text{and}\quad g^m:=3f^m-2f^{m-1},\quad m=1,\dotsc,N,
\end{equation}
and write \eqref{ab-new} as
\begin{equation}
\label{ab-new1}
\dot u^n+  A v^n=g^n,\quad n=7,\dotsc,N.
\end{equation}

For $n\geqslant 11,$ to take advantage of the properties of the multiplier \eqref{mu2},
we consider method \eqref{ab-new1} with $n$ replaced by $n-j,$
multiply it by $\mu_j, j=1,2,3,4,$  and subtract the resulting relations from \eqref{ab-new1},
to obtain
\begin{equation}
\label{ab-new2}
\dot u^n-\sum_{j=1}^4\mu_j \dot u^{n-j} + A \Big (v^n-\sum_{j=1}^4\mu_j  v^{n-j}\Big )=g^n-\sum_{j=1}^4\mu_j  g^{n-j},
\quad n=11,\dotsc,N.
\end{equation}

Now, we subtract and add the term $\hat c_\star(3\dot u^n-2 \dot u^{n-1})$ from and  to the left-hand side 
of \eqref{ab-new2}, and subsequently test the relation by 
\[\tau(3\dot u^n-2 \dot u^{n-1})=\sum_{i=0}^7 \alpha_i v^{n-7+i},\]
 to obtain 
 \begin{equation}
\label{abg31-n}
\tau  I_n+\hat c_\star\tau|3\dot u^n-2 \dot u^{n-1}|^2+\Big \langle \sum_{i=0}^7 \alpha_i v^{n-7+i},  v^n-\sum_{j=1}^4\mu_j  v^{n-j}\Big \rangle =\tau G_n,\quad n=11,\dotsc,N,
\end{equation}
with
 \begin{equation*}\label{2.a7-n}
 \left\{
 \begin{aligned}
&I_n:= \Big (\dot u^n-\sum_{j=1}^4\mu_j \dot u^{n-j}-\hat c_\star(3\dot u^n-2 \dot u^{n-1}),3\dot u^n-2\dot u^{n-1}\Big ),\\
&G_n:= \Big (g^n-\sum_{j=1}^4\mu_j  g^{n-j},3\dot u^n-2\dot u^{n-1}\Big ).
 \end{aligned}
 \right.
\end{equation*}

In view of \eqref{A-kappa-delta}, the pair of polynomials $(\kappa-\hat c_\star \delta, \delta)$
satisfies the A-stability condition \eqref{A} in Lemma \ref{Le:Dahl}; let us denote by
${\widehat G}=(\hat g_{ij})\in \Re^{4,4}$ the corresponding positive definite symmetric matrix entering 
into the analogue to \eqref{G} for this pair of polynomials.

With the notation $\mathcal{V}^n:=(v^{n-6},\dotsc,v^{n})^\top,$ 
$\bm{\dot U}^n:=(\dot u^{n-3},\dotsc,\dot u^{n})^\top,$
and the norms $\|\cdot\|_G$ and  $|\cdot|_{\widehat G},$ given, in analogy to  \eqref{G-norms},
by
 \begin{equation*}\label{G-norms2}
 \|\mathcal{V}^n\|_G^2=\sum_{i,j=1}^7g_{ij}\langle v^{n-7+i},v^{n-7+j}\rangle,
 \quad  |\bm{\dot U}^n|_{\widehat G}^2=\sum_{i,j=1}^4{\hat g}_{ij} (\dot u^{n-4+i},\dot u^{n-4+j}),
\end{equation*}
we have
\begin{equation}\label{G-n}
\Big \langle \sum_{i=0}^7 \alpha_i v^{n-7+i},  v^n-\sum_{j=1}^4\mu_j  v^{n-j}\Big \rangle
\geqslant \|\mathcal{V}^n\|_G^2-\|\mathcal{V}^{n-1}\|_G^2
\end{equation}
and
\begin{equation}\label{G-n2}
I_n\geqslant |\bm{\dot U}^n|_{\widehat G}^2-|\bm{\dot U}^{n-1}|_{\widehat G}^2;
\end{equation}
cf.\ \eqref{G-stab}.

Now, in view of \eqref{G-n} and \eqref{G-n2}, relation \eqref{abg31-n} yields
\begin{equation}\label{2.87-n}
\|\mathcal{V}^n\|_G^2-\|\mathcal{V}^{n-1}\|_G^2+\tau \big (|\bm{\dot U}^n|_{\widehat G}^2-|\bm{\dot U}^{n-1}|_{\widehat G}^2\big )
+\hat c_\star\tau|3\dot u^n-2 \dot u^{n-1}|^2\leqslant \tau G_n.
\end{equation}

Furthermore, the terms involving the forcing term can be estimated
by elementary inequalities in the form
\begin{equation*}
\begin{aligned}
G_n&\leqslant \Big|g^n-\sum_{j=1}^4\mu_j  g^{n-j}\Big|\, |3\dot u^n-2 \dot u^{n-1}|\\
&\leqslant \frac 1{4\hat c_{\star}}\Big|g^n-\sum_{j=1}^4\mu_j  g^{n-j}\Big|^2+
\hat c_{\star}|3\dot u^n-2 \dot u^{n-1}|^2.
\end{aligned}
\end{equation*}
Using this estimate in \eqref{2.87-n} and summing  over $n,$ from $n=11$ to $n=m$,  we obtain
\begin{equation*}\label{2.9-sab2}
\|\mathcal{V}^m\|_G^2-\|\mathcal{V}^{10}\|_G^2+\tau \big (|\bm{\dot U}^m|_{\widehat G}^2-|\bm{\dot U}^{10}|_{\widehat G}^2\big )
 \leqslant  \frac \tau{4\hat c_{\star}}\sum_{n=11}^m\Big|g^n-\sum_{j=1}^4\mu_j  g^{n-j}\Big|^2,
\end{equation*}
and easily see that
\begin{equation*}\label{2.10-sab2}
\|\mathcal{V}^m\|_G^2+\tau |\bm{\dot U}^m|^2_{\widehat G}
 \leqslant \|\mathcal{V}^{10}\|_G^2+ \tau|\bm{\dot U}^{10}|_{\widehat G}^2+C\tau\sum_{n=6}^m|f^n|^2,
 \quad m=11,\dotsc,N.
\end{equation*}
From this estimate, we infer that
\begin{equation*}
\|v^m\|^2+\tau |\dot u^m|^2
 \leqslant C\sum_{j=4}^{10} \|v^j\|^2+ C\tau \sum_{j=7}^{10} |\dot u^j|^2+C\tau\sum_{n=6}^m|f^n|^2,
 \quad m=11,\dotsc,N,
\end{equation*}
and thus
\begin{equation}\label{2.11-sab2}
\|v^m\|^2+\tau |\dot u^m|^2
 \leqslant C\sum_{j=3}^{10} \|u^j\|^2+ C\tau \sum_{j=7}^{10} |\dot u^j|^2+C\tau\sum_{n=6}^m|f^n|^2,
 \quad m=11,\dotsc,N.
\end{equation}

Let us denote by $E_m$ the square root of the quantity on the right-hand side of \eqref{2.11-sab2}.
Then, for $ \ell=11,\dotsc, m\leqslant N,$  \eqref{2.11-sab2} yields $\|v^\ell\| \leqslant E_m,$ whence
\[\|u^\ell\| \leqslant \frac 23\|u^{\ell-1}\|+\frac 13E_m,\quad \ell=11,\dotsc, m.\]
Iterating from $\ell=11$ to $\ell=m,$ we obtain
\[\|u^m\| \leqslant \Big (\frac 23\Big )^{m-10}\|u^{10}\|+\frac 13\sum_{j=0}^{m-11}\Big (\frac 23\Big )^jE_m,\]
and thus
\begin{equation}\label{2.12-sab2}
\|u^m\| \leqslant\|u^{10}\|+E_m, \quad m=11,\dotsc,N.
\end{equation}
Now, \eqref{2.11-sab2} and \eqref{2.12-sab2} yield 
\begin{equation}\label{2.13-sab2}
\|u^m\|^2+\tau |\dot u^m|^2
 \leqslant C\sum_{j=3}^{10} \|u^j\|^2 + C\tau \sum_{j=7}^{10} |\dot u^j|^2 +C\tau\sum_{n=6}^m|f^n|^2,
 \quad m=11,\dotsc,N.
\end{equation}

In view of \eqref{2.13-sab2}, to complete the proof of \eqref{stab-abg2-n}, it suffices to show that
\begin{equation}\label{start-stab}
 \|u^m\|^2+\tau |\dot u^m|^2\leqslant c\sum_{j=0}^6 \|u^j\|^2+ c\tau \sum_{\ell=6}^m |f^\ell|^2,\quad m=7,8,9,10.
 \end{equation}
This can be done via elementary inequalities; cf. \cite[Appendix]{AFKL} and \cite{ACYZ:21}. Testing \eqref{ab-new}
for $n=7$ by $\dot u^7$ and using \eqref{BDF-oper}, we have
\[ |\dot u^7|^2+\frac {3\alpha_7}\tau\|u^7\|^2= -\frac 3\tau\sum_{i=0}^6 \alpha_i \langle u^7,  u^i \rangle
+\frac 2\tau\sum_{i=0}^7 \alpha_i \langle u^6,  u^i \rangle + (3f^7-2 f^6,  \dot u^7 ).\]
Estimating the terms on the right-hand side by the Cauchy--Schwarz and the weighted arithmetic--geometric 
mean inequalities, we can hide the terms involving  $|\dot u^7|^2$ and $\|u^7\|^2/\tau$ to the left-hand side,
and easily obtain \eqref{start-stab} for $m=7.$ Then, using \eqref{start-stab} for $m=7,$  we similarly obtain
the desired result for  $m=8$, and subsequently also for $m=9,10$.

\subsection{Proof of Corollary \ref{Co:stab}}
We write the characteristic polynomial $\alpha$ of the WSBDF7 method \eqref{ab-3}
in the form $\alpha(\zeta)=(\zeta-1)\tilde \alpha(\zeta)$ with $\tilde  \alpha(\zeta)=\tilde  \alpha_6\zeta^6+\dotsb+\tilde  \alpha_0.$
The difference quotients
$\partial_\tau u^m:=(u^m-u^{m-1})/\tau$ satisfy then the equations
\begin{equation}
\label{eq:quotients1}
\sum_{i=0}^6\tilde  \alpha_i \partial_\tau u^{n+i}=\dot u^{n+6},\quad n=1,\dotsc,N-6.
\end{equation}
Since the roots of the polynomial $\tilde \alpha$ lie in the open unit disk, 
the rational function
\[\frac 1{\zeta^6\tilde \alpha(1/\zeta)}=\frac 1{\tilde  \alpha_6+\tilde  \alpha_5\zeta+\dotsb+\tilde  \alpha_0\zeta^6}\]
is holomorphic in a disk of radius larger than $1$ centered at the origin. Thus, Taylor expansion about $0$ yields 
\begin{equation*}
\frac1{\zeta^6\tilde \alpha(1/\zeta)} = \sum_{n=0}^\infty \gamma_n \zeta^n, \ \ |\zeta|\leqslant 1, \quad\text{where}\quad
|\gamma_n|\leqslant c\gamma^n \ \text{ with } \gamma<1.
\end{equation*}
An obvious consequence of this expansion are the relations
\begin{equation}
\label{eq:quotients2}
\tilde \alpha_6\gamma_0=1\quad\text{and}\quad  \sum_{j=0}^6  \tilde \alpha_{6-j}\gamma_{n-j}=0,\quad 
n\in \N,\quad\text{where}\quad \gamma_{-6}=\dotsb=\gamma_{-1}=0.
\end{equation}
To solve the linear difference equation \eqref{eq:quotients1}, we consider the corresponding equations with
$n$ replaced by $n-6-j, j=0,\dotsc,n-7,$ multiply them by $\gamma_j$ and sum over $j.$
In view of \eqref{eq:quotients2}, this leads to
\begin{equation}
\label{eq:quotients3}
\partial_\tau u^n = -\sum_{i=1}^6  \sum_{j=7}^{6+i}  \tilde \alpha_{6-i}\gamma_{n-j}  \partial_\tau u^{j-i}
+\sum_{j=0}^{n-7}  \gamma_j \dot u^{n-j},\quad n=7,\dotsc,N.
\end{equation}
Since $|\gamma_j|\leqslant c\gamma^j$ with  $\gamma<1,$
from \eqref{eq:quotients3} and \eqref{eq:quotients2},  we obtain
\begin{equation}
\label{eq:quotients4}
|\partial_\tau u^n|\leqslant C\sum_{i=1}^6 | \partial_\tau u^i|+C\max_{7\leqslant \ell \leqslant n}|\dot u^\ell|,
\quad n=7,\dotsc,N.
\end{equation}
The asserted estimate \eqref{stab-abg2-nn}  is an immediate consequence of \eqref{eq:quotients4} and
\eqref{stab-abg2-n}.


\section{Error estimates}\label{Se:err-est}
Error estimates are easily established by combining the stability and consistency
of the method.

\begin{proposition}[Error estimates]\label{Pr:err-est}
Assume that the solution $u$ of \eqref{ivp} is sufficiently smooth
and that the starting approximations $u^i\in V$ to $u(t_i),i=0,\dotsc,6,$ 
are sufficiently accurate, namely,
\begin{equation}\label{conv1}
|u(t_i)-u^i|+\tau^{1/2}\|u(t_i)-u^i\|\leqslant C\tau^7, \quad i=0,\dotsc,6.
\end{equation}
Then, we have the error estimate
\begin{equation}\label{conv3}
|u(t_n)-u^n|+\tau^{1/2}\|u(t_n)-u^n\|\leqslant C\tau^7, \quad n=0,\dotsc,N,
\end{equation}
with a constant $C$ independent of the time step $\tau$.
\end{proposition}

\begin{proof}
Let $d^\ell, \ell=7,\dotsc,N,$ denote the consistency error of the seven-step WSBDF method \eqref{ab-3} for the initial value problem \eqref{ivp},
the amount by which the exact solution misses satisfying \eqref{ab-3},
\begin{equation}\label{conc1}
\tau d^{n+7}=\sum_{i=0}^7 \alpha_i u(t_{n+i})+3\tau A u(t_{n+7})-2\tau A u(t_{n+6})-3\tau f^{n+7} +2\tau f^{n+6},
\end{equation}
$n=0,\dotsc,N-7,$ that is,
\begin{equation}\label{conc2}
\tau d^{n+7}=\sum_{i=0}^7 \alpha_i u(t_{n+i})-3\tau u'(t_{n+7})+2\tau u'(t_{n+6}).
\end{equation}

An immediate consequence of the fact that the seven-step WSBDF method
is a linear combination of two methods of order $7,$ namely the
seven-step BDF method and the shifted seven-step BDF method,
is that its order is $7$, i.e.,
\begin{equation}\label{order}
\sum_{i=0}^7 i^\ell\alpha_i=\ell(3\cdot 7^{\ell-1}-2\cdot 6^{\ell-1}),\quad \ell=0,1,\dotsc,7;
\end{equation}
actually, the consistency error of the seven-step WSBDF method is a linear
combination of the consistency errors of the seven-step BDF and shifted seven-step BDF methods.
Therefore, by Taylor expanding about $t_n$ in \eqref{conc2}, we see that, due to the order conditions 
\eqref{order}, leading terms of order up to $7$ cancel, and we obtain
\begin{equation*}
\begin{split}
\tau d^{n+7}&=\frac 1{7!}\Bigg [ \sum\limits^7_{i=0} \alpha_i
\int_{t_n}^{t_{n+i}}(t_{n+i}-s)^7u^{(8)}(s)\, \d s
-21\tau\int_{t_n}^{t_{n+7}}(t_{n+7}-s)^6u^{(8)}(s)\, \d s\Bigg.\\
&\quad\quad\quad\Bigg.+14\tau\int_{t_n}^{t_{n+6}}(t_{n+6}-s)^6u^{(8)}(s)\, \d s\Bigg ],
\end{split}
\end{equation*}
$n=0,\dotsc,N-7.$ Thus, under obvious regularity requirements,
we obtain the desired optimal order consistency estimate
\begin{equation}
\label{cons6}
\max_{7\leqslant \ell\leqslant N}\|d^\ell \|_\star \leqslant C\tau^7.
\end{equation}

Subtracting the numerical method \eqref{ab-3} from the consistency relation \eqref{conc1},
we see that the errors $e^\ell:=u(t_\ell)-u^\ell,\ell=0,\dotsc,N,$ satisfy the error equation
\begin{equation}\label{ab-err}
\sum_{i=0}^7 \alpha_i e^{n+i}+3\tau A e^{n+7}-2\tau A e^{n+6}=\tau d^{n+7},\quad n=0,\dotsc,N-7.
\end{equation}

Now, the stability estimate \eqref{stab-abg2} for the error equation
\eqref{ab-err} in combination with the consistency estimate \eqref{cons6} and our assumption
\eqref{conv1} on the starting approximations lead to the asserted error estimate \eqref{conv3}.
\end{proof}

\section{Numerical results}\label{Se:numerics}
We applied the WSBDF7 method \eqref{ab} to
initial and boundary value problems for the equation
\begin{equation}
\label{heat-eq}
u_t-\varDelta u+u=f\quad\text{in}\quad \varOmega\times [0,T],
\end{equation}
with  $\varOmega=(-1,1)^2$ and $T=1$, subject to \emph{periodic} boundary conditions.
In  space, we discretized by the spectral collocation method  with
the Cheby\-shev--Gauss--Lobatto points.

We numerically verified the  theoretical results including convergence orders in  the discrete $L^2$-norm.
We express the space discrete approximation $u_I^{n}$ in terms of its values
at Chebyshev--Gauss--Lobatto  points,
\[u_I^{n}(x,y)=\sum^{N_x}_{i=0}\sum^{N_y}_{j=0}u_{ij}^{n}\ell_i(x)\ell_j(y),
\quad \ell_i(x)=\prod_{\substack{j=0\\ j\ne i}}^{N_x}\frac{x-x_j}{x_i-x_j},\]
where  $u_{ij}^{n}:=u_I^{n}(x_i,y_j)$ at the mesh points  $(x_i,y_j)$.
Here, $-1=x_0<x_1<\dotsb<x_{N_x}=1$ and $-1=y_0<y_1<\dotsb<y_{N_y}=1$
are nodes of Lobatto quadrature rules.
In order to test the temporal error, we fix $N_x = N_y = 20$;
the spatial error is negligible  since the spectral collocation method converges exponentially;
see, e.g., \cite[Theorem 4.4, \textsection{4.5.2}]{STW:2011}.

\begin{example}[Periodic boundary conditions]\label{ex:6.1}
{\upshape
Here, the initial value and the forcing term were
 chosen such that the exact solution of equation \eqref{heat-eq} is
\begin{equation*}
\label{periodic}
u(x,y,t)=(t^8+1)\cos(\pi x)\cos(\pi y),\quad -1\leqslant x,y \leqslant 1,\  0\leqslant t \leqslant 1.
\end{equation*}
For this case, we present in Table \ref{table:1} the $L^2$-norm of the errors as well as the corresponding
convergence orders (rates) for the WSBDF7 scheme \eqref{ab}.

\begin{table}[!ht]
\begin{center}
\caption{Example \ref{ex:6.1}: The discrete $L^2$-norm errors and numerical convergence orders with $N_x=N_y=20$.}
\begin{tabular}{|c||c|r||c|c||c|c|}
\hline
$\tau$ &        $\vartheta=1$    &   Rate  \ \    & $\vartheta=3$     & Rate       &$\vartheta=10$   &   Rate \\
\hhline{|=======|}
  1/20&    1.7973e-07  &         & 5.0600e-07  &            & 1.2445e-06  &         \\
  1/40&    2.2409e-09  &  $6.3256$ & 4.0273e-09  &6.9732     & 1.3104e-08  &6.5694   \\
  1/80&    1.3982e-10  &  $4.0025$ &3.1624e-11  & 6.9926     & 1.0811e-10  & 6.9215    \\
  1/160&    2.5086e-09  & $-4.1653$ & 2.4183e-13  & 7.0309     &7.6240e-13  & 7.1477     \\
  \hline
    \end{tabular}\label{table:1}
  \end{center}
\end{table}
}
\end{example}

\bibliographystyle{amsplain}


\end{document}